\documentclass[11pt,reqno]{amsart}
\usepackage[dvipsnames]{xcolor}
\usepackage{setspace}
\usepackage[hang,flushmargin,symbol*]{footmisc}
\usepackage{amsmath}
\usepackage{amsthm}
\usepackage{amssymb}
\usepackage{mathtools}
\usepackage{enumitem}
\usepackage{calc}
\usepackage{graphicx}
\usepackage{caption}
\usepackage[labelformat=simple,labelfont={}]{subcaption}
\usepackage{tikz}
\usetikzlibrary{decorations.markings}
\usetikzlibrary{arrows,shapes,positioning}
\usepackage{url}
\usepackage{array}
\usepackage{graphicx}
\usepackage{color}
\usepackage{mathrsfs}

\usepackage{verbatim}

\usepackage{subfiles}

\usepackage[colorinlistoftodos]{todonotes}

\usepackage{dashbox}

\theoremstyle{definition}
\newtheorem{theorem}{Theorem}[section]
\newtheorem{lemma}[theorem]{Lemma}

\newtheorem{corollary}[theorem]{Corollary}
\newtheorem{proposition}[theorem]{Proposition}

\newtheorem{definition}[theorem]{Definition}

\newtheorem{remark}[theorem]{Remark}

\newtheorem{construction}[theorem]{Construction}

\usepackage{amsfonts}
\usepackage{amsmath}
\usepackage{amssymb}
\usepackage{mathrsfs}
\usepackage{tikz}
\usepackage{tikz-cd}
\usepackage{hyperref}
\usepackage{ mathdots }
\usepackage{dutchcal}
\usepackage{verbatim}
\usepackage{amsmath,amsthm}
\usepackage{extarrows}

\newcommand{\stquot}[1]{{[}#1{]}}

\newcommand{\RR}{\mathbb{R}}

\usepackage{xparse}

\newcommand{\Sub}[1]{\ensuremath{\text{Sub}_{#1}}}

\usepackage{multirow}

\usepackage[normalem]{ulem}

\newcommand{\Mone}{\overline{M}_{1, 2}}
\newcommand{\DR}{\text{DR}_\Gamma}

\newcommand{\trop}[1]{#1^{trop}}

\newcommand{\GS}[1]{\overline{M}_{g, n}^\dagger(#1)}
\newcommand{\GSG}[2]{\overline{M}_{#2}^\dagger(#1)}

\newcommand{\msout}[1]{\text{\sout{\ensuremath{#1}}}}

\newcommand{\ZZ}{\mathbb{Z}}
\newcommand{\CC}{\mathbb{C}}
\newcommand{\NN}{{\mathbb{N}}}
\newcommand{\QQ}{{\mathbb{Q}}}
\newcommand{\OO}{\mathcal{O}}

\newcommand{\Aff}{{\mathbb{A}}}
\newcommand{\PP}{\mathbb{P}}
\newcommand{\GG}{\mathbb{G}}

\newcommand{\Nl}[1]{{N_{#1}^{\ell}}}

\newcommand{\Spec}{{\text{Spec}\:}}

\newcommand{\Hom}{\text{Hom}}

\newcommand{\gp}[1]{{#1^{gp}}}

\newcommand{\vir}[1]{{[#1]^{vir}}}
\newcommand{\lvir}[1]{{[#1]^{\ell vir}}}

\newcommand{\cal}[1]{\mathcal{#1}}
\newcommand{\lccx}[1]{\mathbb{L}^{\ell}_{#1}}

\newcommand{\Cl}[1]{{C_{#1}^\ell}}
\newcommand{\Clstrict}[1]{{C_{#1}^{\msout{\ell}}}}
\newcommand{\C}[1]{C_{#1}}

\newcommand{\DNC}[1]{{\widetilde{M}_{#1}}}

\newcommand{\Tl}[1]{{T^{\ell}_{#1}}}

\newcommand{\Log}{{\mathcal{L}}}

\newcommand{\lpb}{{\arrow[dr, phantom, very near start, "\ulcorner \ell"]}}
\newcommand{\lpbstrict}{{\arrow[dr, phantom, very near start, "\ulcorner \msout{\ell}"]}}
\newcommand{\pb}{{\arrow[dr, phantom, very near start, "\ulcorner"]}}

\newcommand{\Ms}{{\overline{M}_{g, n}}}

\newcommand{\Mspecific}[1]{\overline{M}_{#1}}

\newcommand{\Mprel}{{\mathfrak{M}_{g, n}}}

\newcommand{\colim}{\text{colim}}

\newcommand{\QCoh}{\text{QCoh}}
\newcommand{\Coh}{\text{Coh}}

\renewcommand{\tilde}[1]{\ensuremath{\widetilde{#1}}}

\renewcommand{\frak}[1]{\ensuremath{\mathfrak{#1}}}

\newcommand{\onlyinsubfile}[1]{#1}
\newcommand{\notinsubfile}[1]{}

\newcommand{\gys}[1]{#1^\dagger}

\usepackage{stmaryrd}

\definecolor{sebgreen1}{rgb}{0.019,0.317,0.149}
\definecolor{sebgreen2}{rgb}{0.784,0.952,0.780}

\newcommand{\AF}[1]{\mathscr{A}_{#1}}

\newcommand{\HH}{K_\circ}
\newcommand{\HHl}{K_\dagger}

\title{{The Log Product Formula in quantum $K$-theory}}
\author{You-Cheng Chou, Leo Herr, Yuan-Pin Lee}
\date{\today}

\begin{document}

\email{chou@math.utah.edu}

\email{herr@math.utah.edu}

\email{yplee@math.utah.edu}

\address{Institute of Mathematics, Academia Sinica, Taipei 10617, Taiwan, and
Department of Mathematics, University of Utah, 	Salt Lake City, Utah 84112-0090, U.S.A.}

\maketitle

\renewcommand{\onlyinsubfile}[1]{}
\renewcommand{\notinsubfile}[1]{#1}

\setcounter{section}{-1}

\begin{abstract}
    We prove a formula expressing the $K$-theoretic log Gromov-Witten invariants of a product of log smooth varieties $V \times W$ in terms of the invariants of $V$ and $W$. 
    The proof requires introducing log virtual fundamental classes in $K$-theory and verifying their various functorial properties. We introduce a log version of $K$-theory and prove the formula there as well. 
\end{abstract}

\section{Introduction}

The present paper proves a \emph{log product formula} for \emph{quantum $K$-theory}, a $K$-theoretic version of Gromov-Witten theory. We refer to the book \cite{ogusloggeom} for background on log geometry, \cite{mythesislogprodfmla} for the basics of log normal cones and the log product formula, and \cite{QK1} for quantum $K$-theory and $K$-theoretic virtual classes without log structure.

Let $V, W$ be log smooth quasiprojective log schemes. 
Write 
$$Q := \GS{V} \times_{\Ms}^\ell \GS{W},$$ 
where $\GS{X}$ is the stack of log stable maps to $X$, and $\times^{\ell}$ the \emph{fs fiber product}, or fiber product in the category of fs log schemes \cite[Corollary III.2.1.6]{ogusloggeom}. We have maps
\[\GS{V \times W} \overset{h}{\to} Q \overset{\tilde{\Delta}}{\to} \GS{V} \times \GS{W}.\]

The stack $Q$ can acquire a $K$-theoretic \emph{log virtual fundamental class} $[Q]^{\ell vir}$ -- otherwise known as a \emph{log virtual structure sheaf} $\OO_Q^{\ell vir}$ -- in two ways: by pulling back that of $\GS{V} \times \GS{W}$ or pushing forward that of $\GS{V \times W}$. The log product formula asserts these are equal:

\begin{theorem} [$=$\ref{thm:logprodfmla}] \label{t:0.1}
The classes
\[h_\ast \lvir{\GS{V \times W}/\Mprel} = \gys{\Delta} \lvir{\GS{V} \times \GS{W}/\Mprel \times \Mprel}\]
are equal in $\HH(Q)$ as well as $\HHl(Q)$.
Here $\HH(Q) := K(\Coh(Q))$ is the Grothendieck group of coherent sheaves on the lis-\'et site; $\HHl(Q)$ is an inverse limit of $\HH$-theories of log blowups of $Q$ defined in Section~\ref{section:loghhltheory}; $\gys{\Delta}$ is the \emph{log Gysin map} defined in Definition~\ref{d:1.2}.
\end{theorem}

Writing 
\[
ev : \GS{V} \to V^n, \quad ev : \GS{W} \to W^n
\]
for the evaluation maps, we get an equality between log quantum $K$-invariants:

\begin{corollary} \label{cor:chiintformoflogprodfmla}
If $\alpha \in \HH(V)^{\otimes n}, \beta \in \HH(W)^{\otimes n}$, the following invariants are equal
\[
 \chi \left( \OO^{\ell vir}_{\GS{V \times W}} \otimes ev^\ast(\alpha) \otimes ev^\ast(\beta) \right) = \chi \left( \gys{\Delta} ( \OO^{\ell vir}_{\GS{V} \times \GS{W}} \otimes ev^\ast\alpha \times ev^\ast\beta) \right).
\]
\end{corollary}

This follows from Theorem~\ref{t:0.1} and the projection formula. Setting the log structures to be trivial, we obtain the ordinary (non-log) versions of the above results.

\begin{corollary} \label{cor:ordinaryprodfmla}
 The ``ordinary'' (non-log) versions of Theorem~\ref{t:0.1} and Corollary~\ref{cor:chiintformoflogprodfmla} also hold.
\end{corollary}

Without log structure, the log virtual class and Gysin map becomes the ordinary virtual class and Gysin map $\Delta^!$ in \cite{QK1}.

\medskip

M.~Kontsevich and Yu.~Manin formulated the product formula in Gromov--Witten theory in terms of cohomological field theories and proved it in genus zero \cite{KMproduct}. K.~Behrend \cite{prodfmla} generalized it to all genera. Extensions to log Gromov--Witten theory in Chow Groups were obtained by F.~Qu and Y.-P.~Lee \cite{leequlogprodfmla}, D.~Ranganathan \cite{dhruvlogprodfmla}, and L.~Herr \cite{mythesislogprodfmla}. Each builds on the proof of Behrend; the present article further adapts Herr's proof to $\HH$-theory.

Log $\HHl$-theory $\HHl(Q)$ introduced in Section~\ref{section:loghhltheory} can be traced back to \cite{logmotives} and \cite{logchowrecentpaper}. Operations in $\HHl$-theory require comparing the inverse systems of log blowups of the source and target. Section~\ref{section:loghhltheory} undertakes this comparison and will be of use in the forthcoming \cite{rendimentodeicontiwisepandharipandeherrmymolcho} and an ongoing project of the current authors. 

One key ingredient in the formulation of the \emph{log product formula} is the \emph{log Gysin map} $\gys{\Delta}$ in $\HH$-theory in Section~\ref{s:1}. This may be refined to $\HHl$-theory in certain situations. The log normal cone and log virtual fundamental class are defined in $K$-theory, parallel to the constructions denoted simply $\Delta^!$ etc. in \cite{mythesislogprodfmla} for Chow groups. We stress that the product formula would \emph{not} be true with the ordinary Gysin map $\Delta^!$.

The main new technical complication absent for Chow groups comes from the pushforward operation in $\HH$-theory (Section~\ref{s:2}). For example, the na\"ive translation of \emph{Costello's formula} in Chow \cite[Theorem 5.0.1]{costello} to $K$-theory is false for pure degrees $d > 1$. Nevertheless, we prove a log version of Costello's formula in $\HH$-theory for pure degree one (cf.\ \cite[Theorem 4.1]{mythesislogprodfmla}, \cite{mycostellogeneralization}). 

\begin{theorem} [$=$ \ref{thm:costellopfwdkthy}]
Consider an fs pullback square
\[\begin{tikzcd}
X' \ar[d, "f'"] \ar[r, "p"] \lpb      &X \ar[d, "f"]      \\
Y' \ar[r, "q"]      &Y,
\end{tikzcd}\]
where $Y, Y'$ are log smooth, $q$ proper birational and $f$ defines a perfect obstruction theory. Endow $f'$ with the pullback perfect obstruction theory. Under additional technical assumptions specified in Theorem~\ref{thm:costellopfwdkthy}, we have
\[p_\ast \lvir{X'/Y'} = \lvir{X/Y} \quad \in \HH(X).\]
\end{theorem} 

The same is later shown in $\HHl$-theory. Along the way, Hironaka's pushforward theorem \cite[Corollary 2 pg 153]{hironakathesis}, as conjectured by A.~Grothendieck, is strengthened to log smooth stacks. We are not aware of a previous proof even for ordinary stacks. Section \ref{section:logprodfmla} applies these tools to obtain the log product formula. 

To our knowledge, the ordinary product formula in  Corollary~\ref{cor:chiintformoflogprodfmla} is also new. The log version is more potent, as one can see with toric varieties. Any pair of toric varieties of the same dimension are related by a roof of log blowups. Log virtual fundamental classes are invariant under log blowups as in \cite{birationalinvarianceabramovichwise}, \cite[Theorem 3.10]{mythesislogprodfmla} (for Chow), and Proposition \ref{prop:logblowuppfwdvfcs} below (for $K$-theory). Hence \emph{all log Gromov-Witten invariants of a fixed dimension with suitable insertions are the same}.
To compute them, one can take $(\PP^1)^n$. The \emph{log product formula} reduces this to the case of $\PP^1$, whose \emph{log Gromov-Witten invariants} are related to \emph{double ramification cycles}. This plan of attack was shown to us independently by Jonathan Wise and Dhruv Ranganathan, and the main obstacle is the complexity of the operations $\gys{\Delta}$.

The product formula is one of a handful of general tools in the Gromov-Witten repertoire. Versions of the degeneration and localization formulas have been obtained in log Gromov-Witten theory in Chow groups and appear within reach for $K$-theory. The present article is a proof of concept, introducing the log Gysin maps and pushforward identities for this ongoing program.

Throughout, we work with fs (fine and saturated) log structures. The primary reason one needs log virtual fundamental classes instead of ordinary ones is that Diagram~\eqref{eqn:cartesiandiagramlogprodfmla} in Section~\ref{section:logprodfmla} is not cartesian in the category of ordinary schemes. The fs pullback often differs from the scheme-theoretic pullback, even on underlying schemes.

\subsection{Conventions and Definitions}

We only consider fs (fine and saturated) log structures. We use the notation $\Log, \Log Y$ to refer to Olsson's stacks $\mathcal{T\!\!}or, \mathcal{T}\!\!or \,Y$ in \cite{logstacks} with $T$-points:
\[\Log(T) := \{\text{a log structure }M_T\text{ on }T.\}\]
\[\Log X(T) := \left\{
\begin{tikzcd}[row sep=small]
\text{a log structure }M_T\text{ on }T  \text{ and}     \\
\text{a map of log algebraic stacks }T \to X
\end{tikzcd}
\right\}.\]

All stacks are assumed Artin algebraic, locally of finite type over $\CC$, quasiseparated in the sense of \cite[04YW]{sta}, and locally noetherian. Those stacks that aren't quasicompact, e.g., $\Log Y, \Mprel$, are quickly replaced by suitable quasicompact open substacks to avoid imposing ``quasicompact support'' conditions on our $K$-theories. These are not algebraic stacks in the sense of \cite{laumonmoretbaillystacks} because they may not adhere to their stricter notion of ``quasiseparatedness.'' By ``log algebraic stack,'' we mean a stack $X$ as above with a map $X \to \Log$. 

The notation $\ulcorner \ell$, $\times^\ell$, $\Cl{}$ denotes the fs pullbacks and log normal cone, instead of the ordinary scheme-theoretic pullbacks or normal cone $\ulcorner, \times, \C{}$. We use $\ulcorner \msout{\ell}, \times^{\msout{\ell}}$, $\Clstrict{}$ if it so happens that the two coincide. These distinctions are subtle but important, as the difference between pullbacks of schemes and log schemes is the reason for the log normal cone. 

Let $\sigma$ be a sharp fs monoid. Write
\[\Aff_\sigma := \Spec \CC[\sigma],\]
\[\AF{\sigma} := [\Aff_\sigma/\Aff_{\gp{\sigma}}],\]
so that
\[\Hom(X, \AF{\sigma}) = \Hom(\sigma, \Gamma(\overline{M}_X)).\]
An Artin cone is an algebraic stack of the form $\Aff_\sigma$ with its natural log structure. An Artin fan is a log algebraic stack with a representable strict \'etale cover by Artin cones. Note the association $\sigma \mapsto \AF{\sigma}$ is contravariant here, as opposed to the covariant conventions in \cite{wisebounded}, \cite{birationalinvarianceabramovichwise}.

\subsection{Acknowledgments}

We would like to thank Lawrence Barrott, Sebastian Bozlee, Christopher Hacon, Adeel Khan, Gebhard Martin, Pat McFaddin, Sam Molcho, Feng Qu, Dhruv Ranganathan and Jonathan Wise for helpful correspondence. Section~\ref{section:loghhltheory} arose from conversations with Molcho, who pointed out Remark~\ref{rmk:logsmoothimpliesirred} and helped remove unnecessary hypotheses on Proposition~\ref{prop:logblowuppfwdvfcs}. Martin contributed Remark~\ref{rmk:costellopuredegreed}. We also thank the mathoverflow community for \cite{371296}, \cite{375138}, \cite{376138}, \cite{377014}. Navid Nabijou's informal \textit{Guided Meditations} seminar helped catch errors in a draft of Section \ref{s:ddrcounterex} and notes by Dhruv Ranganathan and David Holmes helped set it right. 

The first and the third author wish to thank Academia Sinica, MoST and Simons Foundation for their supports.
The second author is grateful to the NSF for partial financial support through RTG Grant \#1840190.
Our remote communication has been facilitated by the free Google Meet on expensive Apple and Microsoft Computers. We would like to thank Google for making its program freely available. 

\section{Log $\HH$-theory} \label{s:1}

We amalgamate \cite{mythesislogprodfmla}, \cite{fengquktheory}, \cite{virtpb} to develop the basic properties of log virtual fundamental classes in $\HH$-theory. 

A locally noetherian algebraic stack $X$ has a lis-\'et site $X_{lis-\acute{e}t}$ consisting of smooth $X$-schemes $T \to X$ with $T$ affine and \'etale $X$-maps $T' \to T$ between them. We define the categories $\QCoh(X), \Coh(X)$ of quasi-coherent and coherent sheaves to be cartesian sections of the stack of sheaves of modules, with or without finite generation hypotheses \cite{sheavesonartinstacks}. Define $\HH(X) := K(\Coh(X))$ to be the group generated by the coherent sheaves on $X$ modulo relations for exact sequences as in \cite[1.3.2]{fengquktheory}, \cite[3A]{vanishingthmsnegkthy}. This group was denoted $G^{naive}_0(X)$ and shown to coincide with the Thomason-Trobaugh definition of $G$-theory under the assumption $D_{qc}(X)$ is compactly generated in \cite[Lemma 5.5]{vanishingthmsnegkthy}.

To use excision sequences, we need our stacks to be quasicompact and quasiseparated \cite{sta}. 

\begin{remark}\label{rmk:qsepdefns}
There are two notions of \emph{quasiseparated} for a stack $X$ in the literature: the Stacks Project requires the diagonal $\Delta_X$ to be quasicompact and quasiseparated, while \cite{logstacks} and \cite{laumonmoretbaillystacks} require that the diagonal $\Delta_X$ be quasicompact and \emph{separated}. We adhere to the weaker notion \cite[04YW]{sta}, and \cite[Theorem 3.2]{logstacks} shows that $\Log S \to S$ quasiseparated in our sense despite \cite[Remark 3.17]{logstacks} pointing out that $\Log S \to S$ is not quasiseparated in the stronger sense.

Given a morphism $X \to \Log Y$, we can find a quasicompact open subset $U \subseteq \Log Y$ through which the morphism factors: $X \to U \subseteq \Log Y$. Put it another way, any morphism $X \to Y$ with $X$ quasicompact factors as $X \to U \to Y$ with $U \to Y$ quasiseparated and log \'etale, $U$ quasicompact, and $X \to U$ strict. 
\end{remark}

\begin{definition} \label{d:1.2}
Assume $X$, $Y$ are quasiseparated and $X$ is quasicompact. The \textit{log (intrinsic) normal cone} of a DM type map $f : X \to Y$ is defined by
\[\Cl{X/Y} := \C{X/\Log Y}. \]
A \textit{log perfect obstruction theory} for $f$ is an embedding $\Cl{X/Y} \subseteq E$ into a vector bundle over $X$. One similarly has a log deformation to the normal cone $\DNC{X/Y}^\ell = \DNC{X/\Log Y}$ \cite{fengquktheory}. We think of each as $X$-stacks, writing $\Cl{f}|_T$ for the pullback $\Cl{f} \times_X T$ along a map $T \to X$. 

Pick a factorization $X \to U \to Y$ with $X \to U$ strict, $U \to Y$ quasiseparated and log \'etale as in Remark \ref{rmk:qsepdefns}. Given a log perfect obstruction theory, define the \textit{log Gysin map} as
\[\gys{f} : \HH(\Log Y) \to \HH(U) \to \HH(X),\]
the composite of restriction to $U$ and the ordinary Gysin map for $X \to U$ using the obstruction theory $\C{X/U} \simeq \Cl{X/Y} \subseteq E$. The \textit{log virtual fundamental class} $\lvir{X/Y}$ is $\gys{f}[\OO_{\Log Y}]$, if $[\OO_{\Log Y}]$ is defined. This operation and notation may similarly be extended to the $\HH$-theory of log stacks over $Y$.

\end{definition}

\begin{remark}\label{rmk:gysinmapdefnbasicprops}

The map $\gys{f} : \HH(\Log Y) \to \HH(X)$ does not depend on the choice of $U \subseteq \Log Y$. The map $\Log Y \to Y$ is locally of finite presentation \cite[Theorem 1.1]{logstacks}, so $\Log Y$ is locally noetherian \cite[01T6]{sta} and the \textit{map} $i : U \subseteq \Log Y$ is quasicompact \cite[01OX]{sta}. The restriction $\HH(\Log Y) \to \HH(U)$ is consequently well-defined because pullback preserves coherence \cite[Lemma 6.5]{sheavesonartinstacks}. 

The ordinary Gysin map $g^!$ for $g : X \to U$ is defined by the usual commutative diagram \cite[(0.2)]{fengquktheory}
\[\begin{tikzcd}
\HH(\C{X/U}) \ar[r, "i_\ast"]        &\HH(\DNC{X/U}) \ar[r, "j^\ast"] \ar[d, "i^!"]         &\HH(U \times \Aff^1) \ar[r]         &0      \\
        &\HH(\C{X/U})       &\HH(U). \ar[u, "\sim"] \ar[u, "\pi^\ast", swap] \ar[l, dashed, "g^!"]
\end{tikzcd}\]
Then $g^! = i^! \circ ``{(j^\ast)^{-1}}" \circ \pi^\ast$. One composes with the inclusion and intersection with the zero section $\HH(\C{X/U}) \to \HH(E) \to \HH(X)$ to get a class in the $\HH$-theory of $X$. 

The above serves as a proxy for the analogous diagram with $\Log Y$ in place of $U$ because we aren't aware of a proof of excision for non-quasicompact stacks. Remark that $\C{X/U} \simeq \Cl{X/Y} \simeq \Cl{X/Y} \times_Y U$, $\DNC{X/U} \simeq \DNC{X/Y}^\ell|_U$ because $U \subseteq \Log Y$ is open. 

Examine the case of $\gys{f}[\OO_V]$, for $V$ a log smooth $U$-stack. 
Take the fs pullback
\[\begin{tikzcd}
W \ar[r] \ar[d] \lpb       &V \ar[d]      \\
X \ar[r]       &U.
\end{tikzcd}\]
An inclusion $\DNC{W/V} \subseteq \DNC{X/Y}|_W$ results, and witness $j^\ast(\OO_{\DNC{W/V}}) = \OO_{V \times \Aff^1}$. Thus $g^![\OO_V] = i^![\OO_{\DNC{W/V}}] = [\OO_{\Cl{W/V}}]$, and virtual fundamental classes are again fundamental classes of log normal cones.  

If $Y$ is log smooth, $\lvir{X/Y}$ is independent of $Y$ but not of the obstruction theory $E$. By this, we mean that a composite $X \to Y' \to Y$ with obstruction theories $E$ for $X \to Y$ and $E'$ for $X \to Y'$ coming from two-term perfect complexes $\cal E^\bullet, \cal E'^\bullet$ with a compatibility datum
\[\lccx{Y'/Y}|_X \to \cal E^\bullet \to \cal E'^\bullet \to \]
will produce the same virtual fundamental class $\lvir{X/Y} = \lvir{X/Y'}$. The proof is as in \cite[Theorem 3.12]{mythesislogprodfmla}. We omit the proof but use the notation $\lvir{X}$ for $\lvir{X/Y}$ if $Y$ is log smooth. 

\end{remark}

\begin{remark}[Artin fans]

We review a construction generalizing the ``skeleton'' or cone over the dualizing complex of an s.n.c. divisor $D \subseteq X$.

The fs log stack $\AF{P}$ is like $``\Spec P"$ for monoids:
\[\Hom_{fs}(X, \AF{P}) = \Hom_{mon}(P, \Gamma(\overline{M}_X)).\]
Surprisingly, the \emph{stack} $\AF{P}$ represents a strict-\'etale \emph{sheaf} among fs log schemes. Put another way, $\AF{P} \to \Log$ represents an \'etale sheaf on $\Log$-schemes. 

Any fine, finite type log scheme $X$ has a locally closed stratification $X = \bigsqcup X_i$ such that $\overline{M}_X|_{X_i}$ is locally constant. Suppose $\overline{M}_X|_{X_i}$ were actually constant, and write $\overline{M}_{X_i} = \Gamma(X_i, \overline{M}_X|_{X_i})$. There are generization maps if $X_i \subseteq \overline{X}_j$: 
\[\overline{M}_{X_i} \to \overline{M}_{X_j}.\]

If $I$ is the category of strata with generizations $X_i \subseteq \overline{X_j}$ as morphisms, this defines a functor
\[I \to (Mon)^{fine} \to \Log_{\text{\'et}}; \quad \quad \quad i \mapsto \overline{M}_{X_i} \mapsto \AF{\overline{M}_{X_i}}.\]
Taking the colimit as \'etale sheaves on $\Log$ gives the \textit{Artin fan} of $X$ \cite{wisebounded}:
\[\colim_I \AF{\overline{M}_{X_i}} \to \Log.\]

If $\overline{M}_{X_i}$ are locally constant but not constant, they entail monodromy representations of the strata $\pi_1(X_i)$. Cover $X$ instead by strict \'etale maps from schemes with no monodromy. Define $\AF{X}$ to be the colimit of the Artin fans of this cover, again as sheaves over $\Log$.  

The same procedure was done with $\Hom(P, \RR_{\geq 0})$ or $\Hom(P, \NN)$ in place of $\AF{P}$ in \cite{kempfknudsenmumfordsaintdonatkkmstoroidalembeddings} and \cite{loggw}. These constructions and their variants have gone by ``generalized cone complexes,'' ``Kato fans,'' ``monoschemes,'' ``tropicalization,'' etc. in the literature. 

The Artin fan need not be functorial in $X$ \cite[5.4.1]{skeletonsfansabramchenmarcusulrischwise}. For a morphism $X \to Y$, define $\AF{X/Y}$ similarly, using $\Log Y$ in place of $\Log$ above. The map $X \to \AF{X/Y}$ is strict and $\AF{X/Y} \to \AF{Y}$ is log \'etale. See \cite{rendimentodeicontiwisepandharipandeherrmymolcho} for more on functoriality of the Artin fan.

\end{remark}

The hypotheses ``log smooth and equidimensional'' for log virtual classes in \cite{mythesislogprodfmla} were redundant:

\begin{remark}\label{rmk:logsmoothimpliesirred}

We claim connected, quasicompact log smooth stacks $X$ are irreducible. In particular, general quasicompact log smooth stacks are equidimensional and have a fundamental class. This makes the extra hypotheses of ``equidimensional'' or ``irreducible'' on connected log smooth stacks in \cite{mythesislogprodfmla} redundant.

The Artin fan $\AF{X}$ of $X$ is quasicompact because $X$ is, so \cite[Theorem 4.6.2]{wisebounded} supplies a subdivision $\tilde{F} \to \AF{X}$ with a strict map $\tilde{F} \to \AF{\NN}^k$ that is necessarily \'etale. Denote by $\tilde{X} := \tilde{F} \times^{\msout{\ell}}_{\AF{X}} X$ the induced log blowup of $X$, which has a smooth map $\tilde{X} \to \AF{\NN}^k$ to a smooth stack $\AF{\NN}^k$ and hence is smooth. This argument extends \cite[Theorem 5.10]{niziol} to quasicompact log algebraic stacks. 

Log blowups are surjective with geometrically connected fibers \cite[Proposition 2.6]{letcohom2}. A closed (or open), surjective map to a connected topological space with connected fibers has connected source. Hence $\widetilde{X}$ is connected and smooth, which implies irreducible. The same holds for its image $X$. Furthermore, $\Log X$ is irreducible because $X \subseteq \Log X$ is a dense open. It results that $[\OO_{\Log X}]$ is well defined. 

\end{remark}

\begin{remark}\label{rmk:vbtorsorsisomkthy}
Consider a short exact sequence
\[E \to C \to D\]
of cone stacks over noetherian $X$. The pullback is an isomorphism:
\[f^\ast : \HH(D) \simeq \HH(C).\]
The proof of \cite[Theorem 5.7]{vanishingthmsnegkthy} shows $G^{naive}(D) \simeq G^{naive}(C)$ for any vector bundle-torsor, and $\HH(\cdot)=G^{naive}_0(\cdot)$. For example, the intersection with the zero section used to define $\gys{f}$ above is an isomorphism. 

In view of \cite[Proposition 3.3, Corollary 3.4]{khankgthyderalgstacks}, this is an application of \cite[Theorem 3.16]{khankgthyderalgstacks} for perfect stacks $X$. 

\end{remark}

\begin{remark}
For any irreducible, finite type scheme $X$, $\HH(X)$ is generated by $[\OO_V]$ for $V \subset X$ subschemes as a module over $\HH(\Spec \CC)$. The proof uses noetherian induction to reduce to the affine case and applies the Jordan-H\"older filtration. 

This observation asserts that the case of $\gys{f}[\OO_V]$ discussed in Remark \ref{rmk:gysinmapdefnbasicprops} determines $\gys{f}$. We aren't aware of a proof for algebraic stacks. 

\end{remark}

\begin{remark}\label{rmk:gysinmapssquarecommute}
Consider an fs pullback square of qcqs stacks
\[\begin{tikzcd}
X' \ar[r, "p"] \ar[d, "f'"] \lpb      &X \ar[d, "f"]      \\
Y' \ar[r, "q"]      &Y,
\end{tikzcd}\]
where $f$ and $q$ are endowed with log perfect obstruction theories $\Cl{f} \subseteq E$, $\Cl{q} \subseteq F$. The two composite Gysin maps are equal
\[\gys{f'} \gys{q} = \gys{p} \gys{f} : \HH(\Log Y) \to \HH(X').\]

This is seen by replacing $\Log Y, \Log Y', \Log X$ with quasicompact open substacks $U_Y, U_{Y'}, U_X$ containing the images of $X'$ and fitting into a commutative diagram
\[\begin{tikzcd}
X' \ar[dr] \ar[rr, bend left=20] \ar[r, dashed]       &U_X \times_{U_{Y}} U_{Y'} \ar[r] \ar[d] \pb      &U_X \ar[d]        \\
        &U_{Y'} \ar[r]         &U_Y,
\end{tikzcd}\]
applying \cite[Proposition 2.5]{fengquktheory} to the pullback square, and \cite[Proposition 2.11]{fengquktheory} to $X' \to U_X \times_{U_{Y}} U_{Y'} \to U_X$, $X' \to U_X \times_{U_{Y}} U_{Y'} \to U_{Y'}$.

\end{remark}

\begin{remark}

If $Y$ isn't log smooth but $V \to Y$ is a map from such a log stack, one can pull back the square and the log perfect obstruction theories to $V$. This explains the utility of log perfect obstruction theories for non-log smooth $Y$. 

\end{remark}

\section{Pushforward Theorems of Hironaka and Costello} \label{s:2}

Unlike the previous section, these pushforward theorems are not simple applications of \cite{fengquktheory} to maps $X \to \Log Y$. Even for a proper birational morphism $q : Y' \to Y$ between log smooth stacks, one might not have $Rq_\ast \OO_{\Log Y'} = \OO_{\Log Y}$. Indeed, a log blowup $Y' \to Y$ results in an open embedding $\Log Y' \subseteq \Log Y$. Nevertheless, analogues of the pushforward theorems of Hironaka and Costello remain true. We also verify the ``birational invariance'' property of \cite{birationalinvarianceabramovichwise} generalized in \cite[Theorem 3.10]{mythesislogprodfmla}.

A DM type map $f : X \to Y$ of algebraic stacks is \textit{birational} \cite[Definition A.1]{hassetthyeonbiratlmorofstacks} if there is a dense open substack $V \subseteq Y$ whose preimage $f^{-1}V \subseteq X$ is dense and on which $f$ restricts to an isomorphism. This definition coincides with the stacks project if $f$ is locally of finite presentation \cite[0BAC]{sta}. Birationality is smooth-local on the target and satisfies the ``3 for 2'' property:

\begin{center}
If $X \overset{f}{\to} Y \overset{g}{\to} Z$ have composite $h = g \circ f$ and two of $f, g, h$ are birational, so is the third. 
\end{center}

Recall Hironaka's pushforward theorem: a proper birational map $p : X \to Y$ of locally finite presentation with $X, Y$ smooth or rational singularities satisfies $Rp_\ast \OO_X = \OO_Y$ \cite[Theorem 1.1]{hironakapfwdthmposchar}, \cite[Theorem 5.10]{kollarmoribook}, \cite[Corollary 2 pg 153]{hironakathesis}. The reduction to rational singularities merely resolves the singularities of each and then resolves the map between them in such a way that $Rp_\ast \OO_X = \OO_Y$.

\begin{lemma}\label{lem:hironakaforlblowups}
Let $p : \widetilde{X} \to X$ be a log blowup of an fs log smooth algebraic stack $X$ over $\CC$. Then 
\[Rp_\ast \OO_{\widetilde{X}} = \OO_X.\]
\end{lemma}

\begin{proof}

The statement is local in $X$, so assume $X$ is a scheme with smooth global chart $X \to \Aff_P$ and $\tilde{X}$ is globally pulled back from a toric blowup $s : \tilde{F} \to \Aff_P$ of toric varieties. This reduces to the map $s : \tilde{F} \to \Aff_P$ itself, which satisfies $Rs_\ast \OO_{\tilde{F}} = \OO_{\Aff_P}$ because $\tilde{F}$ and $\Aff_P$ have rational singularities and $s$ is proper birational. 

\end{proof}

\begin{remark}
A separated DM stack has finite intertia stacks, hence a coarse moduli space:

A proper unramified map $q$ is finite: \cite[02V5]{sta} shows $q$ is locally quasifinite, \cite[01TJ]{sta} or \cite[01TD]{sta} that $q$ is quasifinite, and \cite[05K0]{sta} finally gives that $q$ is finite. A separated DM stack is defined to have proper unramified diagonal, which is then finite. 

\end{remark}

\begin{proposition}[``Hironaka's Pushforward Theorem'']\label{prop:hironakapfwdthm}
Consider a proper birational morphism $p : X \to Y$ of DM type and locally finite presentation between log smooth algebraic stacks over $\CC$. Then 
\[Rp_\ast \OO_X = \OO_Y.\]
\end{proposition}

\begin{proof}

The statement is smooth-local in $Y$, so assume $Y$ is an affine scheme. Take log blowups of $X, Y$ which are smooth \cite[Theorem 5.10]{niziol} and apply Lemma \ref{lem:hironakaforlblowups} to reduce to the case where $X, Y$ have smooth underlying scheme. This implies $X$ is a separated DM stack, and we must now prove Hironaka's pushforward theorem in this case. 

The coarse moduli space $\pi : X \to \overline{X}$ exists, is a finite map, and is compatible with flat base change \cite{bconradkeelmori}. Then $R^i\pi_\ast \OO_X = 0$  by finiteness \cite[03QP]{sta} and $\pi_\ast \OO_X = \OO_{\overline{X}}$ by the universal property of $\pi$ applied to maps to $\Aff^1$ as discussed after \cite[Theorem 1.1]{bconradkeelmori}.

The coarse space $\overline{X}$ has at worst finite quotient singularities \cite[Lemma 2.2.3]{abramovich-vistoli}, which are rational singularities by \cite{kovacsquotientsingsarerational} or  \cite{boutotquotientsingsarerational}. Choose a proper birational map $t : \widetilde{X} \to \overline{X}$ with smooth source and $Rt_\ast \OO_{\widetilde{X}} = \OO_{\overline{X}}$. Then $\widetilde{X} \to Y$ is proper birational between smooth schemes, so Hironaka's original pushforward theorem applies.

\end{proof}

The interested reader can make sense of ``log rational singularities'' and generalize the proposition.

\begin{corollary}\label{cor:hironakapfwdkfundclass}
In the setting of Proposition \ref{prop:hironakapfwdthm}, pushforward identifies fundamental classes:
\[p_\ast [\OO_X] = [\OO_Y] \quad \in \HH(Y).\]
\end{corollary}

Hironaka's pushforward theorem details when fundamental classes are closed under pushforward. What about log virtual fundamental classes? 

\begin{proposition}\label{prop:logblowuppfwdvfcs}

Let $X \to F$ be a strict morphism of DM type to an Artin fan and $\widetilde{F} \to F$ a proper, birational, DM-type morphism. Write $\widetilde{X}$ for the pullback
\[\begin{tikzcd}
\widetilde{X} \ar[r] \ar[d] \lpbstrict       &\widetilde{F} \ar[d]      \\
X \ar[r]       &F.
\end{tikzcd}\]
This setup allows $\widetilde{X} \to X$ to be a log blowup, a root stack, or composites of such.

Suppose $f : X \to Y$ is a morphism to a \textit{log smooth} algebraic stack equipped with a log perfect obstruction theory $\Cl{X/Y} \subseteq E$ and equip $\widetilde{X} \to Y$ with the induced log perfect obstruction theory. Giving $\widetilde{X}$ the induced log perfect obstruction theory, we have
\[p_\ast \lvir{\widetilde{X}/Y} = \lvir{X/Y} \quad \in \HH(X)\]
\end{proposition}

\begin{proof}

Apply compatibility of pushforward and the Gysin map \cite[Proposition 2.4]{fengquktheory} and Corollary \ref{cor:hironakapfwdkfundclass} to the strict pullback square on Artin fans:
\[p_\ast [\OO_{\C{\widetilde{X}/\widetilde{F}}}] = [\OO_{\C{X/F}}] \quad \in \HH(X).\]
Note that $\C{\widetilde{X}/\widetilde{F}} = \Cl{\widetilde{X}}$, and the same for $X/F$. 
Consider the map of short exact sequences of cone stacks
\[\begin{tikzcd}
\Tl{Y}|_{\widetilde{X}} \ar[r] \ar[d]       &\Cl{\widetilde{X}/Y \times \widetilde{F}} \ar[r, "\widehat{r}"] \ar[d, "\widehat{t}"] \pb      &\Cl{\widetilde{X}} \ar[d, "t"]         \\
\Tl{Y}|_X  \ar[r]     &\Cl{X/Y \times F} \ar[r, "r"]      &\Cl{X}
\end{tikzcd}\]
from \cite[Proposition 2.5]{mythesislogprodfmla}. Remark that $\Cl{\widetilde{X}/Y \times \widetilde{F}} \simeq \Cl{\widetilde{X}/Y}$, etc. After pulling back the bottom row to $\widetilde{X}$, the leftmost vertical map $\Tl{Y}|_{\widetilde{X}} \to \Tl{Y}|_X$ becomes an isomorphism; thus the right square is a pullback. Remark \ref{rmk:vbtorsorsisomkthy} attests $r^\ast, \widehat{r}^\ast$ are isomorphisms, so the commutative square
\[\begin{tikzcd}
\HH(\Cl{\widetilde{X}/Y}) \ar[r, "\widehat{t}_\ast"]       &\HH(\Cl{X/Y})      \\
\HH(\Cl{\widetilde{X}}) \ar[r, "t_\ast"] \ar[u, "\widehat{r}^\ast"]\ar[u, swap, "\sim"]         &\HH(\Cl{X}) \ar[u, "r^\ast"] \ar[u, swap, "\sim"]
\end{tikzcd}\]
coming from compatibility of pullback and pushforward and the equalities $r^\ast[\OO_{\Cl{X}}] = [\OO_{\Cl{X/Y}}]$ etc. imply 
\[\widehat{t}_\ast [\OO_{\Cl{\widetilde{X}/Y}}] = [\OO_{\Cl{X/Y}}]. \]
Composing with the inclusions into the obstruction theories and the Gysin map of the zero section of the obstruction theories, we get our result. 

\end{proof}

The rest of this section concerns Costello's formula, which will be half of our proof of the log product formula. 

\begin{construction}\label{const:artinfanbiratl}

Suppose $f : X \to Y$ is a DM type map between quasicompact log algebraic stacks and $Y$ is log smooth. Quasicompact log algebraic stacks have smooth-locally connected log strata, so \cite[Proposition 3.2.1, Proposition 3.3.2]{wisebounded} produces an Artin fan $Y \to \AF{Y}$ and a relative Artin fan $\AF{X/Y}$ for the pair, both quasicompact:
\[\begin{tikzcd}
X \ar[r] \ar[dr]       &W \ar[r] \ar[d] \lpbstrict      &\AF{X/Y} \ar[d] \ar[r]       &\Log^1 \ar[d]            \\
        &Y \ar[r]      &\AF{Y} \ar[r]         &\Log.
\end{tikzcd}\]    
The map $\AF{X/Y} \to \AF{Y}$ is log \'etale, $W$ is log smooth, and $X \to W$ is strict. All the maps are DM type because Olsson showed $\Log^1 \to \Log$ is and the maps $\AF{Y} \to \Log$, $\AF{X/Y} \to \Log^1$ are representable by construction. The Artin fan of a fine log algebraic stack is locally noetherian. The map $Y \to \AF{Y}$ is smooth and thus noetherian; the same argument shows $W$ is noetherian. 

\end{construction}

Costello's original pushfoward formula \cite[Theorem 5.0.1]{costello} is incorrect as stated; see \cite{mycostellogeneralization}. We prove a $\HH$-theoretic version of Costello's corrected pushforward formula in pure degree one:

\begin{theorem}[``Costello's pushforward formula in $\HH$-theory'']\label{thm:costellopfwdkthy}
Consider an fs pullback square of DM type maps between locally noetherian, locally finite type log algebraic stacks over $\CC$:
\[\begin{tikzcd}
X' \ar[d, "f'"] \ar[r, "p"] \lpb      &X \ar[d, "f"]      \\
Y' \ar[r, "q"]      &Y.
\end{tikzcd}\]

Suppose $f$ has a log perfect obstruction theory $\Cl{X/Y} \subseteq E$ and endow $f'$ with the induced log perfect obstruction theory $\Cl{X'/Y'} \subseteq \Cl{X/Y}|_{X'} \subseteq E|_{X'}$. Assume $Y', Y$ are log smooth, $X$ is quasicompact, and $q$ is proper birational. Then
\[p_\ast \lvir{X'/Y'} = \lvir{X/Y} \quad \in \HH(X).\]
\end{theorem}

\begin{proof}

The proof is a global verion of the argument in \cite[Theorem 4.1]{mythesislogprodfmla}. Construction \ref{const:artinfanbiratl} furnishes us with a factorization $X \to W \to Y$ with $X \to W$ strict and $W \to Y$ log \'etale. Take the fs pullback:
\[\begin{tikzcd}
X' \ar[r] \ar[d] \lpbstrict      &X \ar[d]      \\
W' \ar[r, "s"] \ar[d] \lpb      &W \ar[d]      \\
Y' \ar[r]      &Y.
\end{tikzcd}\]

\textbf{Claim:} $s : W' \to W$ is proper birational.

The schematic fiber product $W \times^{sch}_Y Y' \to W$ is proper and $W' \to W \times^{sch}_Y Y'$ is finite \cite[Proposition III.2.1.5]{ogusloggeom}. 

Birationality of $s$ is smooth-local in $W$ and $Y$, so assume both are affine schemes. Find a commutative square 
\[\begin{tikzcd}
\tilde{W} \ar[r] \ar[d]       &W \ar[d]      \\
\tilde{Y} \ar[r]       &Y
\end{tikzcd}\]
with horizontal maps log blowups such that $\tilde{W} \to \tilde{Y}$ is integral \cite[Theorem]{blowupintegralfkato}. Take fs pullbacks $\tilde{W}' := W' \times^\ell_W \tilde{W}, \tilde{Y}' := Y' \times^\ell_Y \tilde{Y}$ to obtain
\[\begin{tikzcd}
\tilde{W}' \ar[r, "\tilde{s}"] \ar[d] \lpb      &\tilde{W} \ar[d]      \\
\tilde{Y}' \ar[r]      &\tilde{Y}.
\end{tikzcd}\]
The 3 for 2 property of birationality and \cite[Proposition 4.3]{niziol} reduce us to showing $\tilde{s}$ is birational and ensure that $\tilde{Y}' \to \tilde{Y}$ is. We conclude by observing the log \'etale and integral $\tilde{W} \to \tilde{Y}$ is flat \cite[Theorem IV.4.3.5(1)]{ogusloggeom}. 

Corollary \ref{cor:hironakapfwdkfundclass} shows $s_\ast [\OO_{W'}] = [\OO_W]$. Remark $\Cl{X/Y} \simeq \C{X/W}$, so we have an ordinary obstruction theory for $X \to W$. Commutativity of Gysin maps and pushforward \cite[Proposition 2.4]{fengquktheory} gives 
\[p_\ast \vir{X'/W'} = \vir{X/W}\]
and this translates to the statement above via the identifications $\Cl{X/Y} \simeq \C{X/W} \subseteq E$, etc. 

\end{proof}

\begin{remark}[Due to G.\ Martin] \label{rmk:costellopuredegreed}
One might wonder whether the analogous equality 
\[p_\ast \lvir{X'/Y'} = d \cdot \lvir{X/Y} \quad \in \HH(X)\]
holds for proper maps $q : Y' \to Y$ that are of pure degree $d$ instead of birational. This is false even for $f = id_Y$ in ordinary $K$ Theory.

Consider the $k$-algebra 
\[
A = k[x_1, x_2, x_3]/(x_1^3 + x_2^3 + x_3^3)
\]
and its spectrum $Z = \Spec A$. Let $X$ be the minimal resolution of $Z$ given by blowing up the ideal $(x_1,x_2,x_3)$:
\[X := {\rm Proj}\Big(A[y_1,y_2,y_3]/(x_1y_2-x_2y_1,x_1y_3-x_3y_1,x_2y_3-x_3y_2) \Big).\]

The composite $q : X \to Y = \Aff^2 = \Spec k[x_1, x_2]$ of the blowup and the natural projection is of pure degree 3. Now compute $R^iq_\ast(\OO_X) = H^i(\OO_X)$ using the \v{C}ech complex of the open cover $U_i=(y_i \neq 0) \subset X, i = 1, 2, 3$:
\[
0 \rightarrow C^0(\mathcal{U},\mathcal{O}_X) \rightarrow C^1(\mathcal{U},\mathcal{O}_X) \rightarrow C^2(\mathcal{U},\mathcal{O}_X) \rightarrow 0,
\]
with cohomology
\[\begin{split}
H^0(X,\mathcal{O}_X) &= A\cdot (1 \oplus 1 \oplus 1), \\
H^1(X,\mathcal{O}_X) &= A \cdot \left( \frac{y_1^2}{y_2y_3}\oplus \frac{-y_2^2}{y_1y_3} \oplus \frac{y_3^2}{y_1y_2}      \right),  \\ H^2(X,\mathcal{O}_X) &=0.
\end{split}\]
Conclude that
\[Rq_\ast \OO_X = [\tilde{A}] - [\tilde{A}] + 0 = 0 \neq 3 \cdot [\OO_Y].\]

\end{remark}

\section{Log $\HHl$-theory}\label{section:loghhltheory}

This section emerged from conversations with Sam Molcho and Jonathan Wise.

Recall that a morphism $\tilde{X} \to X$ of log algebraic stacks is called a \textit{log blowup} if, strict \'etale locally on $X$, it is the fs pullback of a subdivision of a toric variety at a coherent monoidal ideal. 

One can equivalently ask that $X \to \Aff_P$ be strict or $X$ be an atomic neighborhood by localizing further, or pull back instead from a subdivision of toric stacks $\AF{\Sigma} \to \AF{P}$. Write $\Sub{X}$ for the category of log blowups $\tilde{X} \to X$ of a log algebraic stack $X$, for example if $X$ is an Artin fan. Log blowups are monomorphisms among fs log stacks and fs fiber products preserve log blowups, so $\Sub{X}$ is a cofiltered preorder. All $X$-morphisms $\tilde{X}_1 \to \tilde{X}_2$ between log blowups of $X$ are themselves log blowups.

Recall that a functor $p : I \to J$ is \textit{initial} if the comma category 
\[(p/j) := \{i \to i' \, | \, p(i) \to p(i') \text{ lies over }j\} \]
is nonempty and connected for each $j \in J$ \cite{nlab:final_functor}. This is the dual to concepts variously called ``final'' and ``cofinal'' in the literature and has nothing to do with initial objects in functor categories. If $p$ is initial and $f : J \to C$ any functor, the natural map 
\[\lim_J f \to \lim_{I} f \circ p\]
is an isomorphism. A subcategory is an \textit{initial system} if the inclusion functor is initial.

A map $f : X \to Y$ of log algebraic stacks induces a functor $f^\ast : \Sub{Y} \to \Sub{X}$ sending $\tilde{Y} \to Y$ to $\tilde{Y} \times^\ell_Y X \to X$. If $f$ is itself a log blowup, the map $f_! : \Sub{X} \to \Sub{Y}$ sending $\tilde{X} \to X$ to the composite $\tilde{X} \to Y$ is a section of $f^\ast$ and each is initial. The key technical observation of this section is that $f^\ast$ is initial for $f$ strict.

\begin{definition}
Define \textit{log $\HHl$-theory} of a log algebraic stack $X$ 
\[\HHl(X) := \lim \limits_{\widetilde{X} \to X} \HH(\widetilde{X})\]
as the inverse limit under pushforwards along log blowups of $X$. It has natural maps $\HHl(X) \to \HH(X)$ and $\HHl(X) \to \HH(\widetilde{X})$ for any log blowup $\widetilde{X} \to X$. 
\end{definition}

\begin{remark}
Given a proper morphism $p : X \to Y$, compose the natural map on limits with the levelwise pushforward
\[\lim_{\tilde{X} \in \Sub{X}} \HH(\tilde{X}) \to \lim_{\tilde{Y} \in \Sub{Y}} \HH(\tilde{Y} \times^\ell_Y X) \to \lim_{\tilde{Y} \in \Sub{Y}} \HH(\tilde{Y})\]
to get a morphism $p_\ast : \HHl(X) \to \HHl(Y)$. If $\Sub{Y} \to \Sub{X}$ is initial, $f$ has a log perfect obstruction theory, and levelwise Gysin Maps are compatible with pushforwards, we will also have Gysin maps $\gys{f} : \HHl(Y) \to \HHl(X)$. A similar definition $\lim_{\tilde{X} \in \Sub{X}} A_\ast(\tilde{X})$ for Chow groups was made in \cite{holmespixtonschmitt}.

Log $K^\dagger$-theory can likewise be defined as the colimit under pullbacks of log blowups as in \cite{logmotives}, \cite{logchowrecentpaper} using $K^\circ$-theory. One could instead take Gysin maps as the transition morphisms, provided you require the total spaces to be smooth and use the canonical obstruction theory of an l.c.i. There are natural variants taking (co)limits over log blowups as well as root stacks. 
\end{remark}

\begin{remark}
Lemma \ref{lem:hironakaforlblowups} ensures that the class $([\OO_{\tilde{X}}])_{\tilde{X} \in \Sub{X}}$ is well-defined in $K_\dagger(X)$ for log smooth $X$ over $\CC$. Note this class pushes forward to the ordinary fundamental class in $\HH(X)$. 
\end{remark}

\begin{definition}\label{def:lvirwelldefpfwd}

Suppose $f : X \to Y$ is equipped with a log perfect obstruction theory and $Y$ is log smooth. Define the \textit{log virtual fundamental class} in $\HHl$-theory to be the sequence
\[(\lvir{\tilde{X}/Y})_{\tilde{X} \in \Sub{X}} \quad \in \HHl(X).\]
Proposition \ref{prop:logblowuppfwdvfcs} verifies this sequence lies in the limit $\HHl(X)$. One can similarly define $\gys{f}[\OO_V]$ for log smooth stacks $V \to Y$ but we don't define this operation in full generality on $\HHl$. 

\end{definition}

\begin{lemma}\label{lem:subdivsofconesaregloballyrefined}
Suppose $E \to F$ is a strict map of quasicompact Artin fans and $\tilde{E} \to E$ a subdivision. There exists a subdivision $\tilde{F} \to F$ such that the pullback refines $\tilde{E}$:
\[\begin{tikzcd}
        &\tilde{F} \times_F^{\msout{\ell}} E \ar[d] \ar[dl, dashed]         \\
\tilde{E} \ar[r]         &E.
\end{tikzcd}\]
\end{lemma}

\begin{proof}

Reduce to the case $F = \AF{\NN}^k$ using \cite[Theorem 4.6.2]{wisebounded} and $E = \AF{\sigma}$ a single Artin cone. The strict map $\AF{\sigma} \to \AF{\NN}^k$ is necessarily an isomorphism onto a subcone. Then \cite[Proposition 4.6]{logderivedmckay} exhibits an initial system of subdivisions of $\AF{\NN}^k$. 

\end{proof}

All log blowups are refined by pulling back subdivisions of the Artin fan $\AF{X}$ of $X$ if $X$ is quasicompact. This is because one can assume $X$ is atomic and the strict map $X \to \AF{P}$ induces an isomorphism $P \simeq \Gamma(\overline{M}_X)$. We stop short of showing all log blowups are pulled back from the Artin fan because doing so would require gluing the above local subdivisions, but Lemma \ref{lem:subdivsofconesaregloballyrefined} refines each local subdivision of a cone of $\AF{X}$ by a global subdivision.

\begin{lemma}\label{lem:strictmapinitialpbsubdivisions}
Suppose $f : X \to Y$ is a strict map of quasicompact log algebraic stacks. The functor $f^\ast : \Sub{Y} \to \Sub{X}$ is initial. 
\end{lemma}

\begin{proof}

Consider a log blowup $\tilde{X} \to X$ that we wish to refine. Assume it is pulled back from a subdivision of $\AF{X}$ because these form an initial system as remarked above. The induced map $\AF{X} \to \AF{Y}$ is strict with quasicompact source and target and Lemma \ref{lem:subdivsofconesaregloballyrefined} concludes.

\end{proof}

\begin{proposition}\label{prop:strlblepistalks}
Suppose a quasicompact DM type morphism $f : X \to Y$ of log algebraic stacks with $Y$ quasicompact induces epimorphisms
\[\overline{M}_{Y, f(\overline{x})} \to \overline{M}_{X, \overline{x}}\]
on stalks at geometric points $\overline{x} \to X$. There is a log blowup $\tilde{Y} \to Y$ such that the fs pullback
\[\begin{tikzcd}
\tilde{X} \ar[r, "\tilde{f}"] \ar[d, "p", swap] \lpb           &\tilde{Y} \ar[d, "q"]      \\
X \ar[r, "f", swap]       &Y
\end{tikzcd}\]
has $\tilde{f}$ strict. 
\end{proposition}

\begin{proof}

\textbf{Case:} First suppose $X, Y$ are schemes. 

Locally on $Y$, one chooses a log blowup making $X \to Y$ $\QQ$-integral \cite[Theorem III.2.6.7]{ogusloggeom}. We can refine these local log blowups by a global one using quasicompactness of $Y$ and assume $f : X \to Y$ is $\QQ$-integral, in particular exact \cite[Remark pg. 58]{quasiunipotentriemhilbillusiekatonakayama}. The stalks of the characteristic monoids for a log blowup are epimorphisms, so the property of epimorphic stalks is preserved. 

Exactness entails a pullback
\[\begin{tikzcd}
\overline{M}_{Y, f(\overline{x})} \ar[r] \pb \ar[d]      &\overline{M}_{X, \overline{x}} \ar[d]     \\
\gp{\overline{M}_{Y, f(\overline{x})}}   \ar[r]    &\gp{\overline{M}_{X, \overline{x}}}
\end{tikzcd}\]
with horizontal arrows both epimorphisms. Epimorphisms of groups are surjections \cite{nlab:epimorphisms_of_groups_are_surjective}, so the top horizontal arrow is a surjection. Exactness ensures it is also an injection, hence an isomorphism.

We now reduce the general statement for log stacks to the above case. Cover $Y$ by a strict smooth map $V \to Y$ from a quasicompact scheme and choose a blowup $\tilde{V} \to V$ pulled back from $\AF{V}$ such that $\tilde{V} \times^{\msout{\ell}}_Y X \to \tilde{V}$ is strict. Find a log blowup $\tilde{Y} \to Y$ refining $\tilde{V} \to V$ using Lemma \ref{lem:strictmapinitialpbsubdivisions} to conclude.

\end{proof}

\begin{corollary}\label{cor:epistalkscofinallblpb}
If $f : X \to Y$ satisfies the hypotheses of Proposition \ref{prop:strlblepistalks}, the functor $f^\ast : \Sub{Y} \to \Sub{X}$ is initial. 
\end{corollary}

\begin{proof}

Use notation as in Proposition \ref{prop:strlblepistalks}. The composite $p_! \tilde{f}^\ast q^\ast = p_! p^\ast f^\ast : \Sub{Y} \to \Sub{X}$ is initial as a composite of functors coming from log blowups and the strict $\tilde{f}$ as in Lemma \ref{lem:strictmapinitialpbsubdivisions}. Hence $f^\ast$ is initial.

\end{proof}

\begin{definition}\label{def:epistalksgysinmap}
If $f : X \to Y$ satisfies the hypotheses of Proposition \ref{prop:strlblepistalks} and furthermore is equipped with a log perfect obstruction theory $\Cl{f} \subseteq E$, we can define $\gys{f}$ on $\HHl$-theory as the composite:
\[\HHl(Y) := \lim_{\tilde{Y} \in \Sub{Y}}\HH(\tilde{Y}) \overset{\gys{\tilde{f}}}{\to} \lim_{\tilde{Y} \in \Sub{Y}} \HH(\tilde{Y} \times^\ell_Y X) \simeq \lim_{\tilde{X} \in \Sub{X}} \HH(\tilde X) =: \HHl(X).\]
The first map is induced from the levelwise Gysin maps on $\tilde{f} : \tilde{Y} \times^\ell_X Y \to \tilde{Y}$, while the isomorphism comes from Corollary \ref{cor:epistalkscofinallblpb} showing $f^\ast : \Sub{Y} \to \Sub{X}$ is initial. 
\end{definition}

\begin{proposition}\label{prop:pfwdgysinmappreservelogvfcsinhhllimitkthy}
Consider an fs pullback square
\[\begin{tikzcd}
X' \ar[r, "p"] \ar[d, "f'"] \lpb      &X \ar[d, "f"]      \\
Y' \ar[r, "q"]      &Y
\end{tikzcd}\]
of DM type maps between log algebraic stacks with $Y', Y$ log smooth. Endow $f$ with a log perfect obstruction theory $\Cl{f} \subseteq E$ and equip $f'$ with the induced log perfect obstruction theory $\Cl{f'} \subseteq \Cl{f}|_{X'} \subseteq E|_{X'}$. 

\begin{enumerate}
    \item If $q$ is proper birational and $X$ is quasicompact, then 
    \[p_\ast \lvir{X'/Y'} = \lvir{X/Y} \quad \in \HHl(X).\] 
    \item\label{item:epistalkspullbacklogvirs} Suppose $q : Y' \to Y$ satisfies the hypotheses of Proposition \ref{prop:strlblepistalks}. Endow $q$ with a log perfect obstruction theory $\Cl{q} \subseteq F$. Then 
    \[\gys{q} \lvir{X/Y} = \lvir{X'/Y'} \quad \in \HHl(X').\]
\end{enumerate}

\end{proposition}

\begin{proof}

The map 
\[t : \HHl(X') \to \lim_{\tilde{X} \in \Sub{X}} \HH(\tilde{X} \times^\ell_X X')\]
sends the sequence $\lvir{X'/Y'} := (\lvir{\tilde{X}'/Y'})_{\tilde{X}' \in \Sub{X'}}$ to the sequence $(\lvir{\tilde{X} \times^\ell_X X'/Y'})_{\tilde{X} \in \Sub{X}}$. Costello's pushforward Theorem \ref{thm:costellopfwdkthy} shows the levelwise pushforward sends this sequence to the $(\lvir{\tilde{X}/Y})_{\tilde{X} \in \Sub{X}} =: \lvir{X/Y} \in \HHl(X)$. 

Similarly, one argues the levelwise Gysin maps send $(\lvir{\tilde{X}/Y})_{\tilde{X} \in \Sub{X}}$ to $(\lvir{\tilde{X} \times^\ell_X X'/Y'})_{\tilde{X} \in \Sub{X}}$ by applying Remark \ref{rmk:gysinmapssquarecommute} and then $t^{-1}$ sends this to $(\lvir{\tilde{X}'/Y'})_{\tilde{X}' \in \Sub{X'}} =: \lvir{X'/Y'}$.

\end{proof}

\begin{remark}\label{rmk:cptibilityofgysinmappfwdhhlandhhthy}
Write $s : \HHl(X) \to \HH(X)$, etc. for the natural projections. Under the assumptions of Proposition \ref{prop:pfwdgysinmappreservelogvfcsinhhllimitkthy},
\[s \gys{q} \lvir{X/Y} = \lvir{X'/Y'} = \gys{q} s \lvir{X/Y},\]
\[s p_\ast \lvir{X'/Y'} = \lvir{X/Y} = p_\ast s \lvir{X'/Y'}.\]
In this sense, the operations $\gys{q}, p_\ast$ defined variously on $\HHl$-theory or $\HH$-theory are compatible. 

\end{remark}


\section{The Log Product Formula}\label{section:logprodfmla}


Let $V, W$ be log smooth quasiprojective schemes throughout this section. Write $\Ms \subseteq \Mprel$ for the open substack of stable curves in the moduli space of all genus-$g$, $n$-marked nodal curves. Equip each with the divisorial log structure from the singular locus, an s.n.c. divisor. The stack $\GS{V}$ of \textit{log stable maps} sends an fs log scheme $T$ to diagrams of fs log schemes
\[\begin{tikzcd}
C \ar[r] \ar[d]     &   V      \\
T
\end{tikzcd},\]
where $C$ is an fs log smooth curve over $T$ with genus $g$ and $n$-markings such that the underlying diagram of schemes is a stable map of curves. We emphasize that $\GS{V}$ denotes the stack $\mathscr{M}(V)$ of Gross-Siebert \cite{loggw} parametrizing \textit{log} stable maps  and \textit{not} the ordinary space of stable maps. 

The stack $\frak D$ consists of diagrams $(C' \leftarrow C \rightarrow C'')$ of genus $g$, $n$-pointed prestable curves over $T$ such that $C \to C' \times C''$ is stable. Equivalently, no unstable rational component of $C$ is contracted under both projections.

We have an fs pullback square:

\begin{equation}\label{eqn:cartesiandiagramlogprodfmla}
\begin{tikzcd}
\GS{V \times W} \ar[r, "h"] \ar[d, "c"] \lpb         &Q \ar[r] \ar[d, "b"] \lpb          &\GS{V}\times \GS{W} \ar[d, "a"]        \\
\frak D \ar[r, "\nu"]         &Q' \ar[r, "\phi"] \ar[d] \lpb         &\Mprel \times \Mprel \ar[d, "s \times s"]       \\
        &\Ms \ar[r, "\Delta"]        &\Ms \times \Ms.
\end{tikzcd}
\end{equation}

One lemma and the study of this diagram in \cite{mythesislogprodfmla} suffice to achieve the log product formula.

\begin{lemma}\label{lem:logflatintegralisomnormalcones}
Provided $q$ is log flat and integral in the fs pullback square
\[\begin{tikzcd}
X' \ar[r] \ar[d, "f'"] \lpb      &X \ar[d, "f"]      \\
Y' \ar[r, "q"]      &Y,
\end{tikzcd}\]
the natural inclusion $\Cl{f'} \subseteq \Cl{f}|_{X'}$ is an isomorphism. 
\end{lemma}

\begin{proof}

The claim is smooth-local in $\Cl{f}$, hence in $X, Y$ by Lemmas 2.15, 2.16 of \cite{mythesislogprodfmla}. Assume $X, Y$ are affine schemes equipped with a global chart by Artin Cones: 
\[\begin{tikzcd}
X \ar[r] \ar[d]       &\AF{\sigma} \ar[d]        \\
Y \ar[r]       &\AF{\tau}. 
\end{tikzcd}\]
Write $W := Y \times^{\msout{\ell}}_{\AF{\tau}} \AF{\sigma}$ as in Construction \ref{const:artinfanbiratl} and take the fs pullback of the factorization:
\[\begin{tikzcd}
X' \ar[r] \ar[d] \lpbstrict      &X \ar[d]      \\
W' \ar[r] \ar[d] \lpb      &W \ar[d]      \\
Y' \ar[r]      &Y.
\end{tikzcd}\]
Then $X \to W$ is strict, $W \to Y$ is log \'etale, and we have a map of short exact sequences of cone stacks
\[\begin{tikzcd}
\Tl{W'/Y'}|_{X'} \ar[r] \ar[d]      &\Clstrict{X'/W'} \ar[r] \ar[d] \pb       &\Cl{X'/Y'} \ar[d]       \\
\Tl{W/Y}|_{X} \ar[r]       &\Clstrict{X/W} \ar[r]         &\Cl{X/Y}.
\end{tikzcd}\]
Log flatness of $q$ ensures the leftmost arrow  becomes an isomorphism after pullback to $X'$ \cite[1.1 (iv)]{logcotangent}, so the right square of cones is cartesian. It suffices to show the map of ordinary normal cones $\Clstrict{X'/W'} \to \Clstrict{X/W}|_{X'}$ is an isomorphism. Note $W' \to W$ is log flat and integral, hence it is flat \cite[Theorem IV.4.3.5(1)]{ogusloggeom} and the result is standard.

\end{proof}

\begin{remark}\label{rmk:rmksonbigfspbdiagram}
The map $\frak D \to \Mprel$ sending a trio $(C' \leftarrow C \to C'')$ of partial stabilizations to their source $C$ is log \'etale \cite[Remark 5.9]{mythesislogprodfmla}. Not only does this equate the normal cones $\Cl{\GS{V \times W}/\frak D} \simeq \Cl{\GS{V \times W}/\Mprel}$, but \cite[Lemma 5.10]{mythesislogprodfmla} identifies the natural obstruction theory on $c$ with the one pulled back from $a$. We similarly equip $b$ with the pulled back obstruction theory from $a$. 

The stabilization map $\Mprel \to \Ms$ is integral and log smooth. To see this, use the smooth cover $\bigsqcup_k \Mspecific{g, n + k} \to \Mprel$ and recognize the composites $\Mspecific{g, n+k} \to \Ms$ as iterations of the universal curve $\Mspecific{g, m+1} \to \Mspecific{g, m}$. Give $\Delta$ the canonical log perfect obstruction theory from the isomorphism $\Cl{\Ms/\Ms \times \Ms} \simeq \Nl{\Ms/\Ms \times \Ms}$ as in \cite[Remark 3.7]{mythesislogprodfmla}. Since $s \times s$ is log flat and integral, Lemma \ref{lem:logflatintegralisomnormalcones} equates $\Cl{Q'/\Mprel \times \Mprel} \simeq \Cl{\Ms/\Ms \times \Ms}|_{Q'}$ and the corresponding Gysin maps $\gys{\phi} = \gys{\Delta}$ are equal. 

The log smooth stabilization map also means $Q'$ is log smooth, and we can write $\lvir{Q/Q'} = \lvir{Q}$ as permitted by Remark \ref{rmk:gysinmapdefnbasicprops}.  

\end{remark}

\begin{theorem}[``The log product formula'']\label{thm:logprodfmla}
The log product formula holds in $\HH$-theory as well as $\HHl$-theory: the classes
\[h_\ast \lvir{\GS{V \times W}} = \gys{\Delta} \lvir{\GS{V} \times \GS{W}}\]
are equal in $\HH(Q)$ as well as $\HHl(Q)$. 

\end{theorem}

\begin{proof}

We claim both sides of the equality compute $\lvir{Q/Q'}$ as in \cite{prodfmla}. Remark \ref{rmk:cptibilityofgysinmappfwdhhlandhhthy} shows it suffices to prove the equality in $\HHl(Q)$. 

Note that $\phi$ induces surjective maps
\[\overline{M}_{\Mprel \times \Mprel, \phi(\overline{x})} \to \overline{M}_{Q', \overline{x}}\]
at geometric points $\overline{x} \to Q'$, so $\gys{q'}$ is well-defined on $\HHl$-theory as in Definition \ref{def:epistalksgysinmap}. Proposition \ref{prop:pfwdgysinmappreservelogvfcsinhhllimitkthy} refines Costello's Formula \ref{thm:costellopfwdkthy} and the compatibility we'll need of Remark \ref{rmk:gysinmapssquarecommute} to $\HHl$-theory. In particular, $\gys{\phi}$ and $\nu_\ast$ both preserve log virtual fundamental classes:
\[\gys{\phi} \lvir{\GS{V} \times \GS{W}} = \lvir{Q/Q'}\]
\[\nu_\ast \lvir{\GS{V \times W}} = \lvir{Q/Q'}.\]
Combining these with the equalities $\lvir{\GS{V \times W}} = \lvir{\GS{V \times W}}$ and $\gys{\phi} = \gys{\Delta}$ from Remark \ref{rmk:rmksonbigfspbdiagram} gives the result.

\end{proof}

\begin{remark}

One may also prove Theorem \ref{thm:logprodfmla} as in \cite{dhruvlogprodfmla}. One replaces $\Mprel \times \Mprel$ by $\Mprel(\AF{V}) \times \Mprel(\AF{W})$, makes similar adjustments to $\frak D, Q'$, and takes judicious log blowups to make the analogue of $s \times s$ flat. Using $\Mprel(\AF{V})$, etc. instead of $\Mprel$ makes the maps $a, c$ strict, so their log virtual fundamental classes agree with the ordinary schematic ones. One applies the usual versions of Costello's formula and compatibility of Gysin maps in squares instead of our log-adapted versions.

\end{remark}


\section{Relative variants and a Counterexample}

We offer a relative variant of Theorem \ref{thm:logprodfmla} as in \cite[\S 2.4]{leequlogprodfmla}. This variant makes it easier to see the necessity of $\gys{\Delta}$ instead of $\Delta^!$ in Theorem \ref{thm:logprodfmla}. We end with a counterexample to the version of the relative variant where $\Delta^!$ replaces $\gys{\Delta}$, justifying our technology. 

Let $V \to S$, $W \to T$ be log smooth, quasiprojective maps. We allow $V, W, S, T$ to be algebraic stacks as long as $V \to S$, $W \to T$ are representable by schemes; the ensuing fibers of $\GS{V/S} \to S$ are stable maps to schemes. Our counterexample uses relative stable maps to $\stquot{\PP^1/\GG_m} \to B\GG_m$, like ``rubber'' maps but without expansions. 

Form the analogue of Diagram \eqref{eqn:cartesiandiagramlogprodfmla}:
\begin{equation}\label{eqn:relativecartesiandiagramlogprodfmla}
\begin{tikzcd}
\GS{V \times W/ S \times T} \ar[r, "h"] \ar[d, "c"] \lpb         &Q \ar[r] \ar[d, "b"] \lpb          &\GS{V/S}\times \GS{W/T} \ar[d, "a"]        \\
\frak D \ar[r, "\nu"]         &Q' \ar[r, "\phi"] \ar[d] \lpb         &\Mprel \times \Mprel \ar[d, "s \times s"]       \\
        &\Ms \ar[r, "\Delta"]        &\Ms \times \Ms.
\end{tikzcd}
\end{equation}

\begin{remark}\label{rmk:conncomponentsdiscdatacontactorder}
Log stable maps $C \to V$ have more discrete data than simply genus $g$ and marked points $n$. They also have contact orders at the marked points and the curve class $\beta$ of their image. 

Let $\Gamma$ contain the genus $g$, the number $n$ of marked points, and the contact orders of a stable map $C \to V \times W$ over $S \times T$. This induces discrete data $\Gamma', \Gamma''$ describing the maps $C' \to V$, $C'' \to W$ obtained by stabilization. 

One way to bundle this data together is in the Artin fan $\AF{\GS{V/S}}$. The map from a log stack $X$ to its Artin fan is geometrically connected by construction, but not necessarily surjective if $X$ is not log smooth. This induces an injection from the connected components of log stable maps to those of its Artin fan, which is more closely related to tropical curves. 

Given $\Gamma', \Gamma''$ discrete data for maps to $V$, $W$, there is at most one $\Gamma$ for maps to $V \times W$ that induces $\Gamma'$ and $\Gamma''$. If $g, n$ differ in $\Gamma'$, $\Gamma''$, there are none.

\end{remark}

\begin{theorem}[=\ref{thm:logprodfmla}$'$]
Fix discrete data $\Gamma$ for a stable map $C \to V \times W$ over $S \times T$ as in Remark \ref{rmk:conncomponentsdiscdatacontactorder}. The log product formula holds in $\HHl$-theory and $\HH$-theory:
\[h_\ast \lvir{\GSG{V \times W/S \times T}{\Gamma}} = \gys{\Delta} \lvir{\GSG{V/S}{\Gamma'} \times \GSG{W/T}{\Gamma''}}\]
in $\HHl(Q)$ and $\HH(Q)$.
\end{theorem}

\begin{proof}

Omitted. For the discrete data, the normal cones themselves $\Cl{X/Y}$ respect disjoint unions in the source.

\end{proof}

\subsection{Counterexample to the ordinary log product formula}\label{s:ddrcounterex}

We claim the ``ordinary product formula'' with $\Delta^!$ in the place of $\gys{\Delta}$ is false:
\[h'_\ast \lvir{\GS{V \times W}} \neq \Delta^! \lvir{\GS{V} \times \GS{W}},\]
where $h'$ is the map
\[\GS{V \times W} \overset{h}{\to} Q \overset{\pi}\to \GS{V} \times_{\Ms}^{sch} \GS{W}\]
to the ordinary scheme-theoretic fiber product. This necessitates the introduction of $\gys{\Delta}$ in Section \ref{s:1}.

\begin{figure} 
    \centering
    \begin{tikzpicture}
    \draw[-] (0, 0) to (0, 3);
    \draw[-] (0, 0) to (4, 0);
    \draw[-, very thick] (4, 0) ..controls (5.5, 5) and (5.5, -1).. (4, 3);
    \draw[-] (0, 3) to (4, 3);
    \draw[-, blue, bend left=3] (0, 1.9) to (4.6, 1.7);
    \draw[-, blue, bend right=3] (0, 1.5) to (4.6, 1.7);
    \draw[-, blue] (0, 2.5) to (4.2, 2.5);
    \node[above, blue] at (2, 3){$\DR$};
    \draw[-, very thick] (0, .7) to (4.2, .7);
    \fill (4.64, 1.7) circle (2pt);
    \fill (4.2, .7) circle (2pt);
    \node[below] at (2, 0){$\Mone$};
    \begin{scope}[shift={(-8, 0)}]
    \draw[-] (0, 0) to (0, 3);
    \draw[-] (0, 0) to (4, 0);
    \draw[-, very thick] (4, 0) ..controls (4.7, 1.7).. (4, 3);
    \draw[-, very thick] (4.7, 0.5) ..controls (4, 1.7).. (4.7, 2.8);
    \draw[-] (0, 3) to (4, 3);
    \node[right] at (4.7, 2.8){$E$};
    \draw[-, blue, bend left=3] (0, 1.9) to (4.2, 1.7);
    \draw[-, blue, bend right=3] (0, 1.5) to (4.2, 1.7);
    \draw[-, blue] (0, 2.5) to (4.2, 2.5);
    \node[above, blue] at (2, 3){$\DR^{\dagger}$};
    \draw[-, very thick] (0, .7) to (4.2, .7);
    \fill (4.28, .7) circle (2pt);
    \node[below] at (2, 0){$\Mone^\dagger$};
    \draw[->] (5.3, 1.5) to (7.2, 1.5);
    \end{scope}
    \end{tikzpicture}
    \caption{The map $\Mone^\dagger \to \Mone$ blowing up the self-intersection point of one of the n.c. divisors. The log structure of each comes from the bolded n.c. divisors. }
    \label{fig:doubleramicycle}
\end{figure}
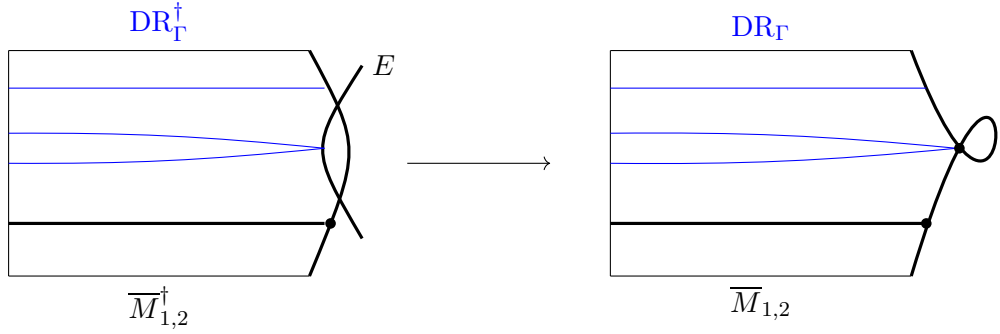

We outline a counterexample due to Dhruv Ranganathan \cite{dhruvlecturenotessubdivisions} and David Holmes, which gives the following:
\[
\begin{split}
&\Delta^! \lvir{\Mone^{\dagger}(\PP^1) \times \Mone^{\dagger}(\PP^1)}
\\
&\hspace{50pt}\neq \pi_\ast \gys{\Delta} \lvir{\Mone^{\dagger}(\PP^1) \times \Mone^{\dagger}(\PP^1)}.
\end{split}
\]
For notational simplicity, write $\DR$ for $\GSG{[\PP^1/\CC^\ast] / B\CC^\ast}{\Gamma}$.
Here we equip $\PP^1$ with divisorial log structure from 0 and $\infty$ and the moduli space consists of relative stable maps 
\[
\begin{tikzcd}
C \ar[r, "f"] \ar[d]       &{[ \PP^1 / \CC^\ast ]} \ar[d]       \\
S \ar[r]       &B\CC^\ast
\end{tikzcd}
\]    
of degree 2 from genus-one, 2-marked curves with contact order 2 at both 0 and $\infty$.

If no confusion may arise, we denote also $\DR$ as its image on $\Mone$, which has 
$\mathbb{Z}/2\mathbb{Z}$-action at all points as their stack structures. $\DR$ can be understood as a degree-three multisection from $\overline{M}_{1,1}$ to $\Mone$ as in Figure \ref{fig:doubleramicycle}. We forbid the two marked points from coming together in $\DR$, a component that is sometimes included in double ramification cycles.

\subsubsection*{\textbf{Counterexample in Chow}}

We first give the computation in Chow theory.
By definition, we have
\[
\Delta^! (\DR \times \DR) = \DR \cdot \DR.
\]

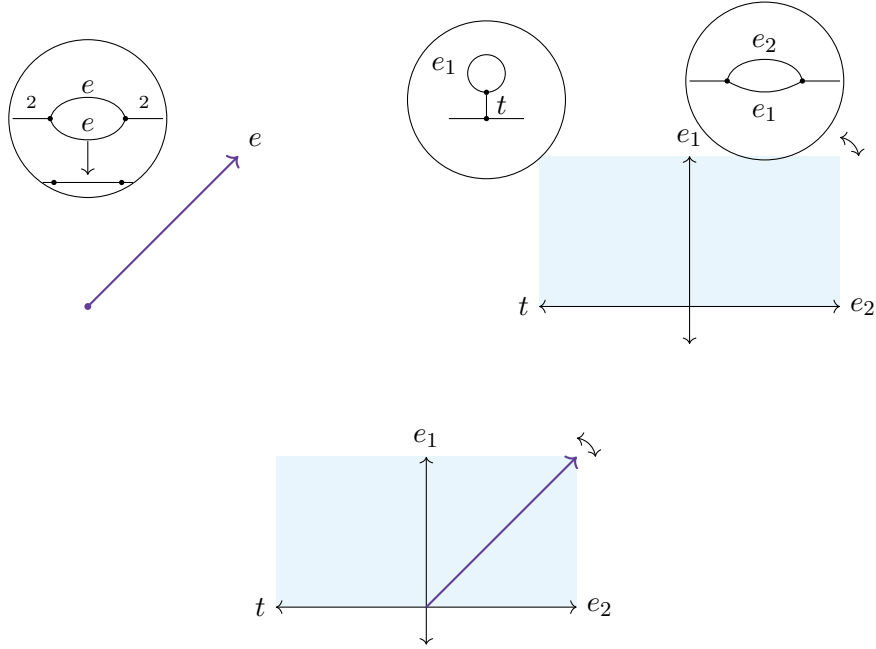
\begin{figure} 
    \centering
    \begin{tikzpicture}
    \begin{scope}[shift={(8, 0)}]
        \fill[Cerulean!8] (-2, 0) rectangle (2, 2);
        \draw[->] (0, 0) -- (2, 0);
        \draw[->] (0, 0) -- (0, 2);
        \draw[->] (0, 0) -- (-2, 0);
        \draw[->] (0, 0) -- (0, -.5);
        \node[right] at (2, 0){$e_2$};
        \node[above] at (0, 2){$e_1$};
        \node[left] at (-2, 0){$t$};
        \draw[<->, bend right=35] (2.25, 2) to (2, 2.25);
    \end{scope}
    \begin{scope}[shift={(5.3, 2.5)}, scale=.5]
        \draw[-] (-1, 0) to (1, 0);
        \draw[-] (0, 0) to (0, .7);
        \node[right] at (0, .35){$t$};
        \node[left] at (-.5, 1.4){$e_1$};
        \draw (0, 1.2) circle (.5);
        \fill (0, 0) circle (2pt);
        \fill (0, .7) circle (2pt);
        \draw (0, .5) circle (2.1);
    \end{scope}
    \begin{scope}[shift={(9, 3)}, scale=.5]
        \draw[-] (1, 0) to (2, 0);
        \draw[-] (-1, 0) to (-2, 0);
        \draw[-, bend left=80] (-1, 0) to (1, 0);
        \draw[-, bend right] (-1, 0) to (1, 0);
        \fill (1, 0) circle (2pt);
        \fill (-1, 0) circle (2pt);
        \node[below] at (0, -.35){$e_1$};
        \node[] at (0, 1){$e_2$};
        \draw (0, 0) circle (2.1);
    \end{scope}
    \begin{scope}[shift={(0, 2.5)}, scale=.5]
        \draw[-] (1, 0) to (2, 0);
        \draw[-] (-1, 0) to (-2, 0);
        \draw[-, bend left=75] (-1, 0) to (1, 0);
        \draw[-, bend right=75] (-1, 0) to (1, 0);
        \fill (1, 0) circle (2pt);
        \fill (-1, 0) circle (2pt);
        \draw (-1.2, -1.7) to (1.2, -1.7);
        \fill (-.9, -1.7) circle (2pt);
        \fill (.9, -1.7) circle (2pt);
        \node[above] at (0, -.6){$e$};
        \node[above] at (0, .4){$e$};
        \node[above] at (1.5, 0){\tiny $2$};
        \node[above] at (-1.5, 0){\tiny $2$};
        \draw (0, 0) circle (2.1);
        \draw[->] (0, -.6) to (0, -1.5);
    \end{scope}
    \draw[->, RoyalPurple, thick] (0, 0) to (2, 2);
    \fill[RoyalPurple] (0, 0) circle (1.3pt);
    \node[above right] at (2, 2){$e$};
    \begin{scope}[shift={(4.5, -4)}]
    \fill[Cerulean!8] (-2, 0) rectangle (2, 2);
        \draw[->] (0, 0) -- (2, 0);
        \draw[->] (0, 0) -- (0, 2);
        \draw[->] (0, 0) -- (-2, 0);
        \draw[->] (0, 0) -- (0, -.5);
        \draw[->, RoyalPurple, thick] (0, 0) to (2, 2);
        \node[right] at (2, 0){$e_2$};
        \node[above] at (0, 2){$e_1$};
        \node[left] at (-2, 0){$t$};
        \draw[<->, bend right=35] (2.25, 2) to (2, 2.25);
    \end{scope}
    \end{tikzpicture}
    \caption{Above left: $\trop{\DR}$. Above right: $\trop{\Mone}$. Below: The subdivision $\Mone^{\dagger,trop}$. Typical members of top-dimensional cones are encircled. Quotient by the $\ZZ/2$ action interchanging $e_1$ and $e_2$. } 
    \label{fig:tropdoubramcycle}
\end{figure}

To compute $\Delta^{\dagger}(\DR \times \DR)$ explicitly, we introduce the following procedure using tropical geometry.

Let $\DR^{trop}$ and $\Mone^{trop}$ be the Artin fans of $\DR$ and $\Mone$ respectively. Think of them as stacky cone complexes with strict maps $\DR \rightarrow \DR^{trop}$ and $\Mone \rightarrow \Mone^{trop}$.
Both $\trop{\DR}$ and $\trop{\Mone}$ have $\ZZ/2$ stack structure from interchanging the edges $e_1$ and $e_2$. The map $\trop{\DR} \to \trop{\Mone}$ is the inclusion of the diagonal in the quadrant spanned by $e_1$ and $e_2$. 
It does not map cones surjectively onto cones, which corresponds to a non-flat map of toric stacks. Take a subdivision $\Mone^{\dagger, trop}$ of $\trop{\Mone}$ by adding the diagonal to the $e_1$-$e_2$ quadrant as in Figure \ref{fig:tropdoubramcycle}. 
This subdivision induces a log blowup 
\[\pi : \Mone^\dagger := \Mone \times_{\trop{\Mone}} \Mone^{\dagger,trop} \to \Mone\] 
such that the factorization $\DR^{\dagger} \to \Mone^{\dagger}$ is strict, where $\DR^{\dagger}$ is the strict transformation of $\DR$. Let $E = \pi^{-1}(p)$ be the exceptional divisor with $\mathbb{Z}/2\mathbb{Z}$-action at all points as stack structure.


Now we have
\[
\Delta^{\dagger} (\DR \times \DR) = \DR^{\dagger} \cdot \DR^{\dagger}
\]
and hence
\[
\begin{split}
\pi_\ast \Delta^{\dagger} (\DR\times \DR) &= \pi_\ast (\DR^{\dagger} \cdot \DR^{\dagger})
\\
&= \pi_\ast \Big((\pi^\ast\DR - E)\cdot(\pi^\ast\DR - E)\Big)
\\
&= \DR \cdot \DR - 
\frac{1}{4}\neq \Delta^! (\DR \times \DR).
\end{split}
\]
Note that $\pi^\ast\DR \cdot E = 0$ by projection formula and $E\cdot E = -\frac{1}{4}$. This gives the inequality for Chow theory.

\subsubsection*{\textbf{Counterexample in $\HH$-theory}}

For $K$-theory, note first that 
\[
\begin{split}
\Delta^![\OO_{\DR} \times \OO_{\DR}] &:= \phi \Big( 
\left[ \OO_{\Delta(\Mone)}\otimes^L \left( \OO_{\DR}\times \OO_{\DR} \right) \right] \Big)
\\
& = \lambda_{-1} (N^\ast_{\DR / \Mone})\cdot \OO_{\DR} ,
\end{split}
\]
where $\lambda_{-1}(F) := \sum_i (-1)^i \Lambda^i F$ is the alternating sum of exterior power of $F$.
We view both $\OO_{\Delta(\Mone)}$ and $\OO_{\DR}\times \OO_{\DR}$ as sheaves on
$\Mone \times \Mone$ and $\phi: K_{\circ}^{\Delta(\Mone)} (\Mone\times \Mone) \simeq K_{\circ}(\Mone)$ is a canonical isomorphism, where $K^Z_{\circ} (X)$ denote the $K$-theory of $X$ supported on $Z$. 
The last equality is the excess intersection formula in $K$-theory \cite[VI.\ Theorem~1.3]{FultonLang}. One can also compute the derived tensor product $\otimes^L$ by taking locally free resolution of $\OO_{\Delta(\Mone)}$ and $\OO_{\DR}$.

The computation of $\Delta^{\dagger}[\OO_{\DR} \times \OO_{\DR}]$ is similar.
\[
\Delta^{\dagger}[\OO_{\DR} \times \OO_{\DR}] := \phi \Big( 
\left[\OO_{\Delta(\Mone^{\dagger})}\otimes^L ( \OO_{\DR^{\dagger}}\times \OO_{\DR^{\dagger}} ) \right] \Big),
\]
Therefore, assuming the third equality below, we have
\begin{equation}\label{e:3}
\begin{split}
& \pi_\ast \Delta^{\dagger} [\OO_{\DR} \times \OO_{\DR}]
\\
= \, &\pi_\ast \phi \Big( 
\left[ \OO_{\Delta(\Mone^{\dagger})} \otimes^L  \left(\OO_{\DR^{\dagger}} \times \OO_{\DR^{\dagger}} \right) \right] \Big) \\
 = \, & \pi_\ast \Big(
\lambda_{-1}(N^\ast_{\DR^{\dagger}/\Mone^{\dagger}}) \cdot \OO_{\DR^{\dagger}}
\Big)
\\
= &- 2\OO_p + \lambda_{-1} (N^\ast_{\DR / \Mone})\cdot \OO_{\DR} \\
= &- 2\OO_p + \Delta^! (\OO_{\DR} \times \OO_{\DR} )\\
\neq \, &\Delta^! [ \OO_{\DR} \times \OO_{\DR} ],
\end{split}
\end{equation}
exactly as claimed.
We are left to show the third equality in \eqref{e:3}.
Write
\[
\OO_{\DR^{\dagger}} = \pi^\ast \OO_{\DR} -\OO_{E} + \OO_{\DR^{\dagger} \cap E}.
\]
The desired equality follows easily from the following equations
\[
\begin{split}
   \pi_\ast( \OO_{\DR^{\dagger}} \cdot \OO_E) &= 
    \pi_\ast\OO_{\DR^{\dagger} \cap E} = \OO_p;
    \\
    \pi_\ast (\OO_{\DR^{\dagger}} \cdot \OO_{\DR^{\dagger} \cap E}) &= 0;
    \\
    \pi_\ast(\OO_E \cdot \OO_{\DR^{\dagger} \cap E}) &= \pi_\ast ( \pi^\ast \OO_p \cdot \OO_{\DR^{\dagger} \cap E} ) = \OO_p \cdot \OO_p = 0;
    \\
    \pi_\ast(\OO_{\DR^{\dagger} \cap E} \cdot \OO_{\DR^{\dagger} \cap E}) &= 0.
\end{split}
\]
The last three equations use the following simple observation.
For any locally free sheaf $V$, 
\[
V \cdot \OO_{\DR^{\dagger} \cap E} = \OO_{\DR^{\dagger} \cap E}^{\oplus {\rm rk}(V)},
\]
since $\DR^{\dagger} \cap E$ is a point.

\bibliographystyle{alpha}
\bibliography{zbib}

\newcommand{\etalchar}[1]{$^{#1}$}
\begin{thebibliography}{KKMSD73}

\bibitem[ACM{\etalchar{+}}15]{skeletonsfansabramchenmarcusulrischwise}
Dan {Abramovich}, Qile {Chen}, Steffen {Marcus}, Martin {Ulirsch}, and Jonathan
  {Wise}.
\newblock {Skeletons and fans of logarithmic structures}.
\newblock {\em arXiv e-prints}, page arXiv:1503.04343, March 2015.

\bibitem[ACWM17]{wisebounded}
Dan Abramovich, Qile Chen, Jonathan Wise, and Steffen Marcus.
\newblock Boundedness of the space of stable logarithmic maps.
\newblock {\em J. Eur. Math. Soc.}, 19:2783--2809, 2017.

\bibitem[AV02]{abramovich-vistoli}
Dan Abramovich and Angelo Vistoli.
\newblock Compactifying the space of stable maps.
\newblock {\em Journal of the American Mathematical Society}, 15(1):27--75,
  2002.

\bibitem[AW18]{birationalinvarianceabramovichwise}
Dan Abramovich and Jonathan Wise.
\newblock Birational invariance in logarithmic gromov–witten theory.
\newblock {\em Compositio Mathematica}, 154(3):595–620, 2018.

\bibitem[{Bar}18]{logchowrecentpaper}
Lawrence~J. {Barrott}.
\newblock {Logarithmic Chow theory}.
\newblock {\em ArXiv e-prints}, October 2018.

\bibitem[Beh99]{prodfmla}
Kai Behrend.
\newblock The product formula for {G}romov-{W}itten invariants.
\newblock {\em J. Algebraic Geom.}, 8(3):529--541, 1999.

\bibitem[Bou87]{boutotquotientsingsarerational}
Jean-Fran{\c{c}}ois Boutot.
\newblock Singularites rationnelles et quotients par les groupes reductifs.
\newblock {\em Inventiones mathematicae}, 88:65--68, 1987.

\bibitem[Con05]{bconradkeelmori}
Brian Conrad.
\newblock Keel-{M}ori theorem via stacks.
\newblock 01 2005.
\newblock \url{https://math.stanford.edu/~conrad/papers/coarsespace.pdf}.

\bibitem[Cos06]{costello}
Kevin Costello.
\newblock Higher genus gromov-witten invariants as genus zero invariants of
  symmetric products.
\newblock {\em Ann. of Math.}, 164(2):561–601, 2006.

\bibitem[CR15]{hironakapfwdthmposchar}
Andre Chatzistamatiou and Kay Rülling.
\newblock Vanishing of the higher direct images of the structure sheaf.
\newblock {\em Compositio Mathematica}, 151(11):2131–2144, 2015.

\bibitem[FL85]{FultonLang}
William Fulton and Serge Lang.
\newblock {\em Riemann-{R}och algebra}, volume 277 of {\em Grundlehren der
  Mathematischen Wissenschaften [Fundamental Principles of Mathematical
  Sciences]}.
\newblock Springer-Verlag, New York, 1985.

\bibitem[GS11]{loggw}
Mark {Gross} and Bernd {Siebert}.
\newblock {Logarithmic Gromov-Witten invariants}.
\newblock {\em ArXiv e-prints}, February 2011.

\bibitem[Hera]{377014}
Leo Herr.
\newblock Do connected algebraic stacks have a smooth cover by a connected
  scheme?
\newblock MathOverflow.
\newblock \url{https://mathoverflow.net/q/377014} (version: 2020-11-20).

\bibitem[Herb]{371296}
Leo Herr.
\newblock Does cohomology and base change hold if supported at a point?
\newblock MathOverflow.
\newblock \url{https://mathoverflow.net/q/371296} (version: 2020-09-09).

\bibitem[Herc]{376138}
Leo Herr.
\newblock G theory localization sequence without ``quasiseparated".
\newblock MathOverflow.
\newblock \url{https://mathoverflow.net/q/376138} (version: 2020-11-11).

\bibitem[Herd]{375138}
Leo Herr.
\newblock {K}/{G}-theory of affine bundles.
\newblock MathOverflow.
\newblock \url{https://mathoverflow.net/q/375138} (version: 2020-11-01).

\bibitem[{Her}19]{mythesislogprodfmla}
Leo {Herr}.
\newblock {The Log Product Formula}.
\newblock {\em arXiv e-prints}, page arXiv:1908.04936, August 2019.

\bibitem[HH09]{hassetthyeonbiratlmorofstacks}
Brendan Hassett and Donghoon Hyeon.
\newblock Log canonical models for the moduli space of curves: the first
  divisorial contraction.
\newblock {\em Transactions of the American Mathematical Society},
  361(8):4471–4489, August 2009.

\bibitem[Hir64]{hironakathesis}
Heisuke Hironaka.
\newblock Resolution of singularities of an algebraic variety over a field of
  characteristic zero: I.
\newblock {\em Annals of Mathematics}, 79(1):109--203, 1964.

\bibitem[HK19]{vanishingthmsnegkthy}
Marc Hoyois and Amalendu Krishna.
\newblock Vanishing theorems for the negative $k$-theory of stacks.
\newblock {\em Ann. K-Theory}, 4(3):439--472, 2019.

\bibitem[HMPW21]{rendimentodeicontiwisepandharipandeherrmymolcho}
Leo {Herr}, Sam {Molcho}, Rahul {Pandharipande}, and Jonathan {Wise}.
\newblock Rendimento dei conti.
\newblock {\em forthcoming}, 2021.

\bibitem[HPS19]{holmespixtonschmitt}
David {Holmes}, Aaron {Pixton}, and Johannes {Schmitt}.
\newblock Multiplicativity of the double ramification cycle.
\newblock {\em Documenta Mathematica}, 24:545--562, 2019.

\bibitem[HW21]{mycostellogeneralization}
Leo {Herr} and Jonathan {Wise}.
\newblock {Costello's pushforward formula: errata and generalization}.
\newblock {\em arXiv e-prints}, page arXiv:2103.10348, March 2021.

\bibitem[IKN05]{quasiunipotentriemhilbillusiekatonakayama}
Luc Illusie, Kazuya Kato, and Chikara Nakayama.
\newblock Quasi-unipotent logarithmic riemann-hilbert correspondences.
\newblock {\em J. Math. Sci. Univ. Tokyo}, pages 1--66, 2005.

\bibitem[IKNU20]{logmotives}
Tetsushi Ito, Kazuya Kato, Chikara Nakayama, and Sampei Usui.
\newblock On log motives.
\newblock {\em Tunis. J. Math.}, 2(4):733--789, 2020.

\bibitem[Kat99]{blowupintegralfkato}
Fumiharu Kato.
\newblock Exactness, integrality, and log modifications.
\newblock {\em arXiv: Algebraic Geometry}, 1999.

\bibitem[{Kha}20]{khankgthyderalgstacks}
Adeel~A. {Khan}.
\newblock {{K}-theory and {G}-theory of derived algebraic stacks}.
\newblock {\em arXiv e-prints}, page arXiv:2012.07130, December 2020.

\bibitem[KKMSD73]{kempfknudsenmumfordsaintdonatkkmstoroidalembeddings}
George Kempf, {Finn Faye} Knudsen, David Mumford, and Bernard Saint-Donat.
\newblock {\em Toroidal Embeddings 1}.
\newblock Lecture Notes in Mathematics. Springer Berlin Heidelberg, 1973.

\bibitem[KM96]{KMproduct}
Maxim Kontsevich and Yuri Manin.
\newblock Quantum cohomology of a product.
\newblock {\em Invent. Math.}, 124(1-3):313--339, 1996.
\newblock With an appendix by R. Kaufmann.

\bibitem[KM98]{kollarmoribook}
J\'{a}nos Koll\'{a}r and Shigefumi Mori.
\newblock {\em Birational geometry of algebraic varieties}, volume 134 of {\em
  Cambridge Tracts in Mathematics}.
\newblock Cambridge University Press, Cambridge, 1998.
\newblock With the collaboration of C. H. Clemens and A. Corti, Translated from
  the 1998 Japanese original.

\bibitem[{Kov}00]{kovacsquotientsingsarerational}
Sándor~J. {Kovács}.
\newblock A characterization of rational singularities.
\newblock {\em Duke Math. J.}, 102(2):187--191, 04 2000.

\bibitem[Lee04]{QK1}
Yuan-Pin Lee.
\newblock Quantum {$K$}-theory. {I}. {F}oundations.
\newblock {\em Duke Math. J.}, 121(3):389--424, 2004.

\bibitem[LMB99]{laumonmoretbaillystacks}
G{\'e}rard Laumon and Laurent Moret-Bailly.
\newblock {\em Champs alg{\'e}briques}.
\newblock Ergebnisse der Mathematik und ihrer Grenzgebiete. 3. Folge / A Series
  of Modern Surveys in Mathematics. Springer Berlin Heidelberg, 1999.

\bibitem[LQ18]{leequlogprodfmla}
Yuan-Pin Lee and Feng Qu.
\newblock A product formula for log {G}romov-{W}itten invariants.
\newblock {\em J. Math. Soc. Japan}, 70(1):229--242, 2018.

\bibitem[{Man}08]{virtpb}
Cristina {Manolache}.
\newblock {Virtual pull-backs}.
\newblock {\em ArXiv e-prints}, May 2008.

\bibitem[Nak17]{letcohom2}
Chikara Nakayama.
\newblock Logarithmic étale cohomology {II}.
\newblock {\em Advances in Mathematics}, 314:663--725, 07 2017.

\bibitem[Niz06]{niziol}
Wies\l{}awa Nizio\l{}.
\newblock Toric singularities: Log-blow-ups and global resolutions.
\newblock {\em J. Algebraic Geom.}, 15(1):1--29, 2006.

\bibitem[{nLa}20a]{nlab:epimorphisms_of_groups_are_surjective}
{nLab authors}.
\newblock epimorphisms of groups are surjective.
\newblock
  \url{http://ncatlab.org/nlab/show/epimorphisms%20of%20groups%20are%20surjective},
  November 2020.
\newblock
  \href{http://ncatlab.org/nlab/revision/epimorphisms%20of%20groups%20are%20surjective/4}{Revision
  4}.

\bibitem[{nLa}20b]{nlab:final_functor}
{nLab authors}.
\newblock final functor.
\newblock \url{http://ncatlab.org/nlab/show/final%20functor}, November 2020.
\newblock \href{http://ncatlab.org/nlab/revision/final%20functor/26}{Revision
  26}.

\bibitem[Ogu18]{ogusloggeom}
Arthur Ogus.
\newblock {\em Lectures on Logarithmic Algebraic Geometry}.
\newblock Cambridge Studies in Advanced Mathematics. Cambridge University
  Press, 2018.

\bibitem[Ols03]{logstacks}
Martin Olsson.
\newblock Logarithmic geometry and algebraic stacks.
\newblock {\em Scientific Annals of the Ecole Normale Sup\'erieure}, 36, 2003.

\bibitem[Ols05]{logcotangent}
Martin Olsson.
\newblock The logarithmic cotangent complex.
\newblock {\em Mathematische Annalen}, 2005.

\bibitem[Ols07]{sheavesonartinstacks}
Martin Olsson.
\newblock Sheaves on {A}rtin stacks.
\newblock {\em J. Reine Angew. Math.}, 603:55--112, 2007.

\bibitem[Qu18]{fengquktheory}
Feng Qu.
\newblock Virtual pullbacks in {$K$}-theory.
\newblock {\em Annales de l'Institut Fourier}, 68(4):1609--1641, 2018.

\bibitem[{Ran}19]{dhruvlogprodfmla}
Dhruv {Ranganathan}.
\newblock {A note on cycles of curves in a product of pairs}.
\newblock {\em arXiv e-prints}, page arXiv:1910.00239, October 2019.

\bibitem[Ran20]{dhruvlecturenotessubdivisions}
Dhruv Ranganathan.
\newblock A {GUI} for logarithmic intersections, July 2020.

\bibitem[SST18]{logderivedmckay}
Sarah Scherotzke, Nicolò Sibilla, and Mattia Talpo.
\newblock On a logarithmic version of the derived mckay correspondence.
\newblock {\em Compositio Mathematica}, 154(12):2534–2585, 2018.

\bibitem[{Sta}18]{sta}
The {Stacks Project Authors}.
\newblock \textit{Stacks Project}.
\newblock \url{http://stacks.math.columbia.edu}, 2018.

\end{thebibliography}

\end{document}